\documentclass[a4paper]{amsart}

\usepackage{amsthm,amsfonts,amsmath, amssymb}

\usepackage{isomath}
\usepackage{lmodern} 
\usepackage[italic]{mathastext}

\usepackage{setspace}
\usepackage{graphicx,float,enumerate,subfigure,tikz}
\usepackage{algpseudocode}
\usepackage{verbatim}
\usepackage{url}
\usepackage{epstopdf}
\usepackage[capitalise]{cleveref}
\newtheorem{theorem}{Theorem}
\newtheorem{corollary}[theorem]{Corollary}
\newtheorem{lemma}[theorem]{Lemma}
\newtheorem{proposition}[theorem]{Proposition}

\newenvironment{claimproof}{\noindent\textit{Proof.}}{\hfill$\square$}

\newtheorem{question}[theorem]{Question}

\theoremstyle{definition}
\newtheorem{definition}[theorem]{Definition}

\theoremstyle{remark} 
\newtheorem{remark}[theorem]{Remark}

\crefname{remark}{Remark}{Remarks}

\theoremstyle{definition}
\newtheorem{claim}{Claim}
\usepackage{etoolbox}
\AtEndEnvironment{proof}{\setcounter{claim}{0}}

\theoremstyle{remark}
\newtheorem{case}{Case}
\usepackage{etoolbox}
\AtEndEnvironment{proof}{\setcounter{case}{0}}
\AtEndEnvironment{claimproof}{\setcounter{case}{0}}

\usepackage{mathtools}

\DeclarePairedDelimiter\floor{\lfloor}{\rfloor}
\DeclareMathOperator{\dist}{dist}

\DeclareMathOperator{\aff}{aff}
\DeclareMathOperator{\conv}{conv}

\DeclareMathOperator{\st}{star}
\DeclareMathOperator{\a-st}{astar}
\DeclareMathOperator{\lk}{link}

\newcommand{\R}{\mathbb{R}}

\newcommand{\A}{\mathcal{A}}

\newcommand{\C}{\mathcal{C}}
\newcommand{\B}{\mathcal{B}}

\newcommand{\St}{\mathcal{S}}
\newcommand{\Lk}{\mathcal{L}}

\crefname{rmk}{Remark}{Remarks}
\crefname{problem}{Problem}{Problems}
\date{\today}
\title{The linkedness of cubical polytopes}

\usepackage{amsaddr}

\author{Hoa T. Bui}
\address{Centre for Informatics and Applied Optimisation, Federation University Australia}
\email{\texttt{h.bui@federation.edu.au}}
\author{Guillermo Pineda-Villavicencio \& Julien Ugon}
\address{Centre for Informatics and Applied Optimisation, Federation University Australia\\School of Information Technology, Deakin University, Australia}
\email{\texttt{julien.ugon@deakin.edu.au}} 
\email{\texttt{work@guillermo.com.au}}

\thanks{Hoa T. Bui is supported by an Australian Government Research Training Program (RTP) Stipend and RTP Fee-Offset Scholarship through Federation University Australia. Julien Ugon's research was partially supported by ARC discovery project DP180100602.}
\keywords{$k$-linked, cube,  cubical polytope, connectivity, separator, linkedness}
\subjclass[2010]{Primary 52B05; Secondary 52B12}

\begin{document}
\begin{abstract} A cubical polytope is a polytope with all its facets being combinatorially equivalent to cubes. The paper is concerned with the linkedness of the graphs of cubical polytopes. A graph with at least $2k$ vertices is $k$-linked if, for every set of $2k$ distinct vertices organised in arbitrary $k$ unordered pairs of vertices, there are $k$ vertex-disjoint paths joining the vertices in the pairs. 

Larman and Mani in 1970 proved that simplicial $d$-polytopes, $d$-dimensional polytopes with all their facets being combinatorially equivalent to simplices, are $\floor{(d+1)/2}$-linked;  this is the maximum possible linkedness given the facts that a $\floor{(d+1)/2}$-linked graph is at least $(2\floor{(d+1)/2}-1)$-connected and that some of these graphs are $d$-vertex-connected but not $(d+1)$-vertex-connected. 

Here we establish that $d$-dimensional cubical polytopes are also $\floor{(d+1)/2}$-linked for every $d\ne 3$; this is again the maximum possible linkedness for such a class of polytopes.
\end{abstract}
\maketitle

\section{Introduction}

The {\it graph} $G(P)$ of a polytope $P$ is the undirected graph formed by the vertices and edges of the polytope. This paper studies the the linkedness of {\it cubical $d$-polytopes}, $d$-dimensional polytopes with all their facets being cubes.  By a cube we mean any polytope that is combinatorially equivalent to a cube; that is, one whose face lattice is isomorphic to the face lattice of a cube.   
 
Denote by $V(X)$ the vertex set of a graph or a polytope $X$. Given sets $A,B$ of vertices in a graph, a path from $A$ to $B$, called an {\it $A-B$ path}, is a (vertex-edge) path $L:=u_{0}\ldots u_{n}$ in the graph such that $V(L)\cap A=\{u_{0}\}$  and $V(L)\cap B=\{u_{n}\}$. We write $a-B$ path instead of $\{a\}-B$ path, and likewise, write $A-b$ path instead of $A-\{b\}$. 

Let $G$ be a graph and $X$ a subset of $2k$ distinct vertices of $G$. The elements of $X$ are called {\it terminals}. Let $Y:=\{\{s_{1},t_{1}\}, \ldots,\{s_{k},t_{k}\}\}$ be an arbitrary labelling and (unordered) pairing of all the vertices in $X$. We say that $Y$ is {\it linked} in $G$ if we can find disjoint $s_{i}-t_{i}$ paths for $i\in [1,k]$, the interval $1,\ldots,k$. The set $X$ is {\it linked} in $G$ if every such pairing of its vertices is linked in $G$. Throughout this paper, by a set of disjoint paths, we mean a set of vertex-disjoint paths. If $G$ has at least $2k$ vertices and every set of exactly $2k$ vertices is linked in $G$, we say that $G$ is {\it $k$-linked}. If the graph of a polytope is $k$-linked we say that the polytope is also {\it $k$-linked}.

Unless otherwise stated, the graph theoretical notation and terminology follow from \cite{Die05} and the polytope theoretical notation and terminology from \cite{Zie95}. Moreover, when referring to graph-theoretical properties of a polytope such as minimum degree, linkedness and connectivity, we mean properties of its graph.

Being $k$-linked imposes a stronger demand on a graph than just being $k$-vertex-connected, or $d$-connected for short. A $k$-linked graph needs to be at least $(2k-1)$-connected, and yet there are $(2k-1)$-connected graphs that are not $k$-linked. The classification of 2-linked graphs \cite{Sey80,Tho80} contextualised for 3-polytopes readily gives  examples of this phenomenon: with the exception of simplicial 3-polytopes, no 3-polytope, despite being  $3$-connected by Balinski's theorem \cite{Bal61}, is 2-linked. However, there is a linear function $f(k)$ such that every $f(k)$-connected graph is $k$-linked, which follows from works of Bollob\'as and Thomason \cite{BolTho96}; Kawarabayashi, Kostochka and Yu \cite{KawKosYu06}; and Thomas and Wollan \cite{ThoWol05}. In the case of polytopes, Larman and Mani \cite[Thm.~2]{LarMan70} proved that every $d$-polytope  is $\floor{(d+1)/3}$-linked, a result that was slightly improved to $\floor{(d+2)/3}$  in \cite[Thm.~2.2]{WerWot11}.  

The first edition of the Handbook of Discrete and Computational Geometry \cite[Problem 17.2.6]{GooORo97-1st} posed the question of whether or not every $d$-polytope is $\floor{d/2}$-linked. This question had already been answered in the negative by Gallivan in the 1970s with a construction of a $d$-polytope that is not $\floor{2(d+4)/5}$-linked; see \cite{Gal85}. A weak positive result however follows from \cite{ThoWol05}: every $d$-polytope with minimum degree at least $5d$ is $\floor{d/2}$-linked.
 
Restricting our attention to particular classes of polytopes gives stronger results. {\it Simplicial $d$-polytopes}, polytopes in which every facet is a simplex, are $\floor{(d+1)/2}$-linked \cite[Thm.~2]{LarMan70}. Since there are simplicial $d$-polytopes that are $d$-connected but not $(d+1)$-connected, the bound of $\floor{(d+1)/2}$ is best possible for this class of polytopes.   Polytopes with small number of vertices were considered in \cite{WerWot11}, where it was shown that $d$-polytopes with $d+\gamma+1$ vertices are $\floor{(d-\gamma+1)/2}$-linked for $0\le \gamma\le (d+2)/5$. 

  In his PhD thesis \cite[Question~5.4.12]{Ron09} Wotzlaw asked whether every cubical $d$-polytope is $\floor{d/2}$-linked. Here we answer the question in the strongest possible way. 

\noindent {\bf Theorem.}  {\it For every $d\ne 3$, a cubical $d$-polytope is $\floor{(d+1)/2}$-linked.   }

Our methodology relies on results on the connectivity of strongly connected subcomplexes of cubical polytopes, whose proof ideas were first developed in \cite{ThiPinUgo18v3}, and a number of new insights into the structure of $d$-cube (\cref{sec:cube}). One obstacle that forces some tedious analysis is the fact that the 3-cube is not 2-linked.
 
In line with the main result of \cite{ThiPinUgo18v3}, where it was proved that a cubical $d$-polytope of minimum degree $\delta$ is $\min\{\delta,2d-2\}$-connected, we wonder if the following is true.

\begin{question}\label{q:cubical-strong-linkedness}
For every $\delta\ne 3$, is a cubical polytope with minimum degree $\delta$ necessarily $\floor{(\delta+1)/{2}}$-linked?
\end{question}

\section{Preliminary results}
\label{sec:preliminary}
This section groups a number of results that will be used in later sections of the paper.

The next two propositions follow from the characterisation of 2-linked graphs carried out in \cite{Sey80,Tho80}. Both propositions also have proofs stemming from arguments in the form of \cref{lem:linear-paths}, a lemma used implicitly in the original proof of Balinski's theorem (\cref{thm:Balinski}) and made explicit in \cite[Thm.~3.1]{Sal67}; for the sake of completeness we give such proofs.

\begin{lemma}[{\cite[Thm.~3.1]{Sal67}}] \label{lem:linear-paths} Let $P$ be a $d$-polytope, and let $f$ be a linear function on $\mathbb{R}^{d}$ satisfying $f(x)>0$ for some $x\in P$. If $u$ and $v$ are vertices of $P$ with $f(u)\ge 0$ and $f(v)\ge 0$, then there exists a $u-v$ path $x_{0}x_{1}\ldots x_{n}$ with $x_{0}=u$ and $x_{n}=v$ such that $f(x_{i})>0$ for $i\in [1,n-1]$.  
\end{lemma}

\begin{theorem}[Balinski {\cite{Bal61}}]\label{thm:Balinski} For every $d\ge 1$, the graph of a $d$-polytope is $d$-connected. 
\end{theorem}
 
Let  $X$ be a set of vertices in a graph $G$. Denote by $G[X]$ the subgraph of $G$ induced by $X$, the subgraph of $G$ that contains all the edges of $G$ with vertices in $X$. Write $G-X$ for $G[V(G)\setminus X]$. A path in the graph is called {\it $X$-valid} if no inner vertex of the path is in $X$. The {\it distance} between two vertices $s$ and $t$ in a graph $G$, denoted $\dist_{G}(s,t)$, is the length of a shortest path between the vertices.  
  
 \begin{definition}[Configuration 3F]\label{def:Conf-3F} Let $X$ be a set of at least four terminals in a 3-cube and let $Y$ be a  labelling and pairing of the vertices in $X$. A terminal of $X$, say $s_{1}$, is in {\it Configuration 3F} if the following conditions are satisfied:
 \begin{enumerate}
 \item[(i)] four vertices of $X$ appear in a 2-face $F$ of the cube;
 \item[(ii)] the terminals in the pair $\left\{s_{1},t_{1}\right\}\in Y$ are at distance two in $F$ (that is, $\dist_{F}(s_{1},t_{1})=2$); and 
\item[(iii)] the neighbours of $t_{1}$ in $F$ are all vertices of $X$.
 \end{enumerate}	
\end{definition}  
 
 Configuration 3F is the only configuration in a 3-cube that prevents the linkedness of a  pairing $Y$ of four vertices, as  \cref{prop:3-polytopes} attests. A sequence $a_{1},\ldots, a_{n}$ of vertices in a cycle is in {\it cyclic order} if, while traversing the cycle, the sequence appears in clockwise or counterclockwise order. It follows that, if pairing $Y:=\left\{\left\{s_{1},t_{1}\right\},\left\{s_{2},t_{2}\right\}\right\}$ of vertices in a 3-cube appears in cyclic order $s_{1},s_{2},t_{1},t_{2}$ in a 2-face, then all the vertices in $Y$ are in Configuration 3F.

\begin{proposition}\label{prop:3-polytopes} Let $G$ be the graph of a 3-polytope and let $X$ be a set of four vertices of $G$. The set $X$ is linked in $G$ if and only if there is no facet of the polytope containing all the vertices of $X$.
\end{proposition}
\begin{proof} Let $P$ be a 3-polytope embedded in $\R^{3}$ and let  $X$ be an arbitrary set of four vertices in $G$. We first establish the necessary condition by proving the contrapositive. Let $F$ be a 2-face containing the vertices of $X$ and  consider a planar embedding of $G$ in which $F$ is the outer face. Label the vertices of $X$ so that they appear in the cyclic order $s_{1},s_{2},t_{1},t_{2}$. Then the paths $s_{1}-t_{1}$ and $s_{2}-t_{2}$ in $G$ must inevitably intersect, implying that $X$ is not linked.

Assume there  is no 2-face of $P$ containing all the vertices of $X$. Let $H$ be a (linear) hyperplane that contains $s_{1}$, $s_{2}$ and $t_{1}$, and let $f$ be a linear function that vanishes on $H$ (this may require a translation of the polytope). Without loss of generality, assume that $f(x)>0$ for some $x\in P$ and that $f(t_{2})\ge 0$. 

First consider the case that $H$ is a supporting hyperplane of a 2-face $F$.  The subgraph $G(F)-\{s_{2}\}$ is connected, and so there is an $X$-valid $L_{1}:=s_{1}-t_{1}$ path on  $G(F)$. Then, use \cref{lem:linear-paths} to find an $L_{2}:=s_{2}-t_{2}$ path in which each inner vertex has positive $f$-value. The paths $L_{1}$ and $L_{2}$ are clearly disjoint. 

Now consider the case that $H$ intersects the interior of $P$. Then there is a vertex in $P$ with $f$-value  greater than zero and a vertex with $f$-value less than zero. Use \cref{lem:linear-paths} to find  an $s_{1}-t_{1}$ path in which each inner vertex has negative $f$-value and an $s_{2}-t_{2}$ path in which each inner vertex has positive $f$-value. 
\end{proof}

The subsequent corollary follows at once from \cref{prop:3-polytopes}. 
\begin{corollary} No nonsimplicial 3-polytope is 2-linked.
\label{cor:nonsimplicial-3polytope}
\end{corollary}

 The same reasoning employed in the proof of the sufficient condition of \cref{prop:3-polytopes} settles \cref{prop:4polytopes}.

\begin{proposition}[2-linkedness of 4-polytopes]\label{prop:4polytopes} Every 4-polytope is 2-linked.
\end{proposition}
\begin{proof} Let $G$ be the graph of a 4-polytope embedded in $\mathbb R^{4}$. Let $X$ be a given set of four vertices in $G$ and let $Y:=\{\{s_{1},s_{2}\},\{t_{1},t_{2}\}\}$ a labelling and pairing of the vertices in $X$. 

Consider a linear function $f$ that vanishes on a linear hyperplane $H$ passing through $X$. Consider the two cases in which either $H$ is  a supporting hyperplane of a facet $F$ of $P$ or $H$ intersects the interior of $P$.

Suppose $H$ is a supporting hyperplane of a facet $F$. First, find an $s_{1}-t_{1}$ path in the subgraph $G(F)-\{s_{2},t_{2}\}$, which is connected by Balinski's theorem. Second, use \cref{lem:linear-paths} to find an $s_{2}-t_{2}$ path that touches $F$ only at $\{s_{2},t_{2}\}$.     

If instead $H$ intersects the interior of $P$ then there is a vertex in $P$ with $f$-value  greater than zero and a vertex with $f$-value less than zero. Use \cref{lem:linear-paths} to find  an $s_{1}-t_{1}$ path in which each inner vertex has negative $f$-value and an $s_{2}-t_{2}$ path in which each inner vertex has positive $f$-value. \end{proof}

The definitions of polytopal complex and strongly connected complex play an important role in the paper. A {\it polytopal complex} $\C$ is a finite nonempty collection of polytopes in $\R^{d}$ where the faces  of each polytope in $\C$ all belong to $\C$ and where polytopes intersect only at faces (if $P_{1}\in \C$ and $P_{2}\in \C$ then $P_{1}\cap P_{2}$ is a face of both $P_{1}$ and $P_{2}$). The empty polytope is always in $\C$. The {\it dimension} of a complex $\C$ is the largest dimension of a polytope in $\C$;  if $\C$ has dimension $d$ we say that $C$ is a {\it $d$-complex}. Faces of a complex $\C$ of largest and second largest dimension are called {\it facets} and {\it ridges}, respectively. If each of the faces of a complex $\C$ is contained in some facet we say that $\C$ is {\it pure}. 

Given  a polytopal complex $\C$ with vertex set $V$ and a subset $X$ of $V$,  the subcomplex of $\C$ formed by all the faces of $\C$ containing only vertices from $X$ is called {\it induced} and is denoted by $\C[X]$.  Removing from $\C$ all the vertices in a subset $X\subset V(\C)$  results in the subcomplex $\C[V(\C)\setminus X]$, which we write as $\C-X$. If $X=\{x\}$ we write $\C-x$ rather than $\C-\{x\}$. We say that a subcomplex $\C'$ of a complex $\C$ is a {\it spanning} subcomplex of $\C$ if $V(\C')=V(\C)$. The {\it graph} of a complex is the undirected graph formed by the vertices and edges of the complex; as in the case of polytopes, we denote the graph of a complex $\C$ by $G(\C)$.  A pure polytopal complex $\C$ is {\it strongly connected} if every pair of facets $F$ and $F'$ is connected by a path $F_{1}\ldots F_{n}$ of facets in $\C$ such that $F_{i}\cap F_{i+1}$ is a ridge of $\C$ for $i\in [1,n-1]$, $F_{1}=F$ and $F_{n}=F'$; we say that such a path is a {\it $(d-1,d-2)$-path} or a {\it facet-ridge path} if the dimensions of the faces can be deduced from the context. 

The relevance of strongly connected complexes stems from a result of Sallee that is described below.

\begin{proposition}[{\cite[Sec.~2]{Sal67}}]\label{prop:connected-complex-connectivity} For every $d\ge 1$, the graph of a strongly connected $d$-complex is $d$-connected. 
\end{proposition}

Strongly connected complexes can be defined from a $d$-polytope $P$. Two basic examples are given by the complex of all faces of $P$, called the {\it complex} of $P$ and denoted by $\C(P)$, and the complex of all proper faces of $P$, called the {\it boundary complex} of $P$ and denoted by $\B(P)$.  For a polytopal complex $\C$, the {\it star} of a face $F$ of $\C$, denoted $\st(F,\C)$, is the subcomplex of $\C$ formed by all the faces containing $F$, and their faces; the {\it antistar} of a face $F$ of $\C$, denoted $\a-st(F,\C)$, is the subcomplex of $\C$ formed by all the faces disjoint from $F$; and the {\it link} of a face $F$, denoted $\lk(F,\C)$, is the subcomplex of $\C$ formed by all the faces of $\st(F,\C)$ that are disjoint from $F$. That is, $\a-st(F,\C)=\C-V(F)$ and $\lk(F,\C)=\st(F,\C)-V(F)$. Unless otherwise stated, when defining stars, antistars and links in a polytope, we always assume that the underlying complex is the boundary complex of the polytope.    
 
 Let   $v$ be a vertex in a $d$-cube $Q_{d}$ and let $v^{o}$ denote the vertex at distance $d$ from $v$, called the vertex {\it opposite} to $v$. The star of a vertex $v$ in the boundary complex of a $d$-cube $Q_{d}$ is the subcomplex $Q_{d}-v^{o}$, the subcomplex induced by $V(Q_{d})\setminus \{v^{o}\}$.
\begin{remark}
\label{rmk:opposite-vertex}
The antistar of $v$ coincides with the star of $v^{o}$. Consequently, the link of $v$ in a $d$-cube $Q_{d}$ is the subcomplex $Q_{d}- \{v, v^{o}\}$.  \end{remark}

Figure~\ref{fig:4-cube} depicts the star and link of a vertex in the 4-cube.

\begin{figure}
\includegraphics{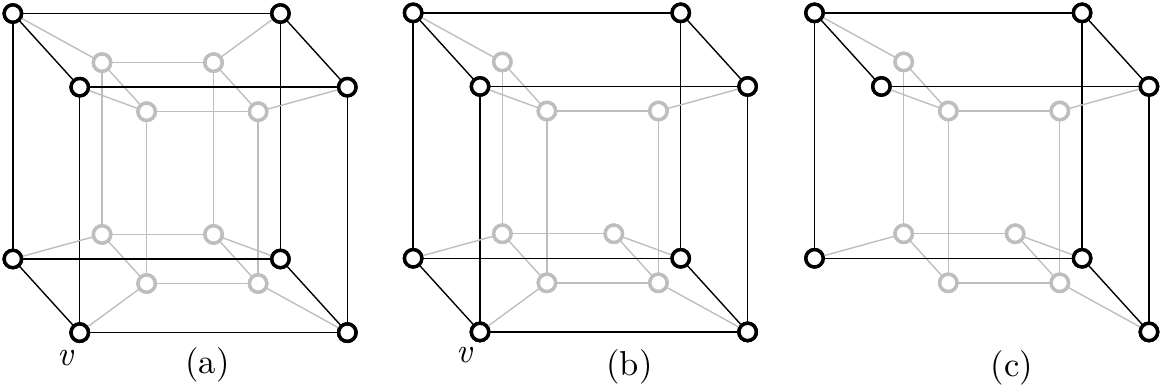}
\caption{Complexes in the 4-cube. {\bf (a)} The 4-cube with a vertex $v$ highlighted. {\bf (b)} The star of the  vertex $v$. {\bf (c)} The link of the vertex $v$.}\label{fig:4-cube} 
\end{figure} 

Some of the aforementioned complexes defined from a $d$-polytope are indeed strongly connected $(d-1)$-complexes, as the next proposition attests. The parts about the boundary complex and the antistar of a vertex already appeared in \cite{Sal67}. A proof of \cref{prop:st-ast-connected-complexes}, which uses the notation and terminology of this paper, was provided in \cite[Prop.~6]{ThiPinUgo18v3}. 

\begin{proposition}[{\cite[Cor.~2.11, Thm.~3.5]{Sal67}}]\label{prop:st-ast-connected-complexes} Let $P$ be a $d$-polytope. Then, the boundary complex $\B(P)$ of $P$, and the star and antistar of a vertex in $\B(P)$, are all strongly connected $(d-1)$-complexes of $P$.
\end{proposition}
   
By considering a point $v'$ in $\R^{d}$ beyond a vertex $v$ of a $d$-polytope $P$ and using \cite[Thm.~5.2.1]{Gru03}, we get a statement similar to \cref{prop:st-ast-connected-complexes} for the link  of a vertex in $\B(P)$: \cref{prop:link-polytope}. We provide all the details, for the sake of completeness. 
  
Following \cite[pp.~78, 241]{Zie95},  we say that a facet $F$ is {\it visible} from a point $v'$ in $\R^{d}\setminus P$ if $v'$ belongs to the open halfspace that is determined by $\aff F$, the affine hull of the facet, and is disjoint from $P$; if instead $v'$ belongs to  the open halfspace that contains the interior of $P$, we say that  the facet is {\it nonvisible} from $v'$.  Further we say that a point $v'$ in $\R^{d}$ is {\it beyond} a face $K$ of $P$ if the facets containing $K$ are precisely those visible from $v'$.  
 
\begin{theorem}[{\cite[Thm.~5.2.1]{Gru03}}]
\label{thm:beneath-beyond} Let $P$ and $P'$ be two $d$-polytopes in $\mathbb{R}^d$, and let $v'$ be a vertex of $P'$ such that $v'\not\in P$ and $P'=\conv (P\cup \{v'\})$. Then
\begin{enumerate}[(i)]
\item a face $F$ of $P$ is a face of $P'$ if and only if there exists a facet  of $P$ containing $F$ that is nonvisible from $v$;

\item if $F$ is a face of $P$ then $F':=\conv(F\cup\{v'\})$ is a face of $P'$ if
 
\begin{enumerate}[(a)]
\item either $v'\in \aff F$;
\item or among the facets of $P$ containing $F$ there is at least one that is visible from $v'$ and at least one that is nonvisible.
\end{enumerate}
\end{enumerate}
Moreover, each face of $P'$ is of exactly one of the above three types.
\end{theorem} 
 
\begin{proposition}[{\cite[Ex.~8.6]{Zie95}}]\label{prop:link-polytope} Let $P$ be a $d$-polytope. Then the link of a vertex in $\B(P)$ is combinatorially equivalent to the boundary complex of a $(d-1)$-polytope.
\end{proposition}
 
 \begin{proof}Let $v$ be a vertex of $P$ and let $v'$ be a point in $\R^{d}\setminus P$ beyond $v$  so that $v'$ is not on the affine hull of any face of $P$. Suppose $P':=\conv (P\cup \{v'\})$.

The facets in the star of $v$ in $\B(P)$ are precisely those that are visible from $v'$, and every other facet of $P$, including the facets in the antistar of $v$ in $\B(P)$, is nonvisible from $v'$.  The link of $v$ is, by definition, the subcomplex of $\B(P)$ induced by the ridges of $P$ that are contained in a facet of the star of $v$, a facet visible from $v'$, and a facet of the antistar of $v$, a facet nonvisible from $v'$. Consequently, according to \cref{thm:beneath-beyond}(i), the ridges in $\lk (v,\B(P))$ are faces of $P'$. Furthermore, for every ridge $R\in \lk (v,\B(P))$, $R':=\conv (R\cup \{v'\})$ is  a facet of $P'$ (\cref{thm:beneath-beyond}(ii-b)), a pyramid over $R$ with apex $v'$; and every facet in the star of $v'$ in $\B (P')$ is one of these pyramids.   Hence, the vertex figure of $P'$ at $v'$, which  is a $(d-1)$-polytope \cite[Sec.~2.1]{Zie95}, is combinatorially equivalent to the link of $v$ in $P$, as desired.  
 \end{proof}

\cref{prop:link-polytope} is exemplified in \cref{fig:link-polytope}. 
\begin{figure}
\includegraphics{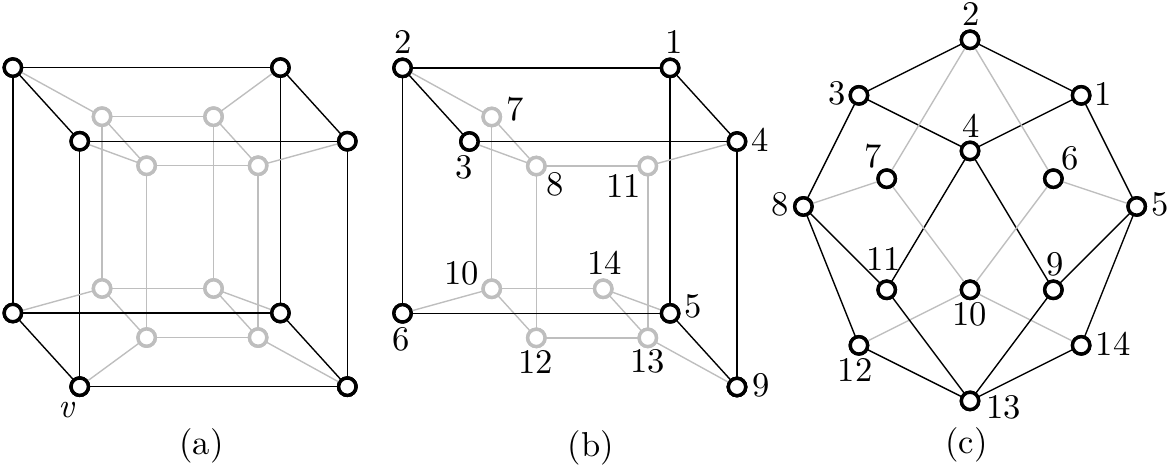}
\caption{The link of a vertex in the 4-cube. {\bf (a)} The 4-cube with a vertex $v$ highlighted. {\bf (b)}  The link of the vertex $v$ in the 4-cube. {\bf (c)} The link of the vertex $v$ as the boundary complex of the rhombic dodecahedron (\cref{prop:link-polytope}).} \label{fig:link-polytope}   
\end{figure}

\section{$d$-cube}  
\label{sec:cube}

In the $d$-cube $Q_{d}$, the facet disjoint from a facet $F$ is denoted by $F^o$, and we say that $F$ and $F^{o}$ is a pair of {\it opposite} facets. 

\begin{definition}[Projection $\pi$]\label{def:projection}
	For a pair of opposite facets $\{F,F^{o}\}$ of $Q_{d}$, define a projection $\pi^{Q_{d}}_{F^{o}}$ from $Q_{d}$ to $F^{o}$ by sending a vertex $x\in F$ to the unique neighbour $x^{p}_{F^o}$ of $x$ in $F^{o}$, and  a vertex $x\in F^{o}$ to itself (that is, $\pi^{Q_{d}}_{F^{o}}(x)=x$); write $\pi^{Q_{d}}_{F^o}(x)=x^{p}_{F^o}$ to be precise, or write $\pi(x)$ or $x^{p}$ if the cube $Q_{d}$ and the facet $F^o$ are understood from the context. 
\end{definition}

We extend this projection to sets of vertices: given a pair $\{F,F^{o}\}$ of opposite facets and a set $X\subseteq V(F)$, the projection $X^{p}_{F^{o}}$ or $\pi^{Q_{d}}_{F^{o}}(X)$ of $X$ onto $F^{o}$ is the set of the projections  of the vertices in $X$ onto $F^{o}$. For an $i$-face $J\subseteq F$, the projection $J^{p}_{F^{o}}$ or $\pi_{F^{o}}^{Q_{d}}(J)$ of $J$ onto $F^{o}$ is the $i$-face consisting of the projections of all the vertices of $J$ onto $F^{o}$. For a pair $\{F,F^{o}\}$ of opposite facets in $Q^{d}$, the restrictions of the projection $\pi_{F^{o}}$ to $F$  and the projection $\pi_{F}$ to $F^{o}$ are bijections.

Let  $Z$ be a set of vertices in the graph of a $d$-cube $Q_{d}$. If, for some pair of opposite facets $\{F,F^o\}$, the set $Z$ contains both a vertex $z\in V (F)\cap Z$ and its projection $z_{F^{o}}^p\in V (F^o)\cap Z$, we say that the pair $\{F,F^o\}$ is {\it associated} with the set $Z$ in $Q_{d}$ and that $\{z,z^p\}$ is an {\it associating pair}.  Note that an associating pair can associate only one pair of opposite facets.

In conjunction with connectivity results around strongly connected complexes in cubical polytopes, the next lemma lies at the core of our methodology. 
 
\begin{lemma}\label{lem:facets-association} Let $Z$ be a nonempty subset of $V(Q_{d})$. Then the number of pairs $\{F,F^o\}$ of opposite facets associated with $Z$ is at most $|Z|-1$. 
\end{lemma} 

\begin{proof}
Let $G:=G(Q_{d})$ and let $Z\subset V(Q_{d})$ with $|Z|\ge 1$ be given. Consider a pair $\{F,F^{o}\}$ of opposite facets. Define a {\it direction} in the cube as the set of the $2^{d-1}$ edges between $F$ and $F^{o}$; each direction corresponds to a pair of opposite facets.  The $d$ directions partition the edges of the cube into sets of cardinality $2^{d-1}$.  (The notion of direction stems from thinking of the cube as a zonotope \cite[Sec.~7.3]{Zie95}) 

A pair of facets is associated with the set $Z$ if and only if the subgraph $G[Z]$ of $G$ induced by $Z$ contains an edge from the corresponding direction.

If a direction is present in a cycle $C$ of $Q_{d}$, then the cycle contains at least two edges from this direction. Indeed, take an edge $e=uv$  on $C$ that belongs to a direction between a pair $\{F,F^{o}\}$ of opposite facets.   After traversing  the edge $e$ from $u\in V(F)$ to $v\in V(F^{o})$, for the cycle to come back to the facet $F$, it must contain another edge from the same direction.   Hence, by repeatedly removing edges from cycles in $G[Z]$ we obtain a spanning forest of  $G[Z]$ that contains an edge for every direction present in $G[Z]$. As a consequence, the number of such directions is at most the number of edges in the forest, which is upper bounded by $|Z|-1$. (A {\it forest} is a graph with no cycles.) \end{proof}

The relevance of the lemma stems from the fact that a pair of opposite facets $\{F,F^{o}\}$ not associated with a given set of vertices $Z$ allows each vertex $z$ in $Z$ to have ``free projection''; that is, for every $z\in Z\cap V(F)$ the projection $\pi_{F^o}(z)$ is not in $Z$, and for $z\in Z\cap V(F^{o})$ the projection $\pi_{F}(z)$ is not in $Z$.

\subsection{Connectivity of the $d$-cube}

We next unveil some further properties of the cube that will be used in subsequent sections.

While it is true that the antistar of a vertex in a $d$-polytope is always a strongly connected $(d-1)$-complex (\cref{prop:st-ast-connected-complexes}), it is far from true that this extends to higher dimensional faces. Refer to \cite[Sec.~3]{ThiPinUgo18v3} for examples of $d$-polytopes in which this extension is not possible. This extension is however possible for the $d$-cube, as shown in \cite[Lem.~8]{ThiPinUgo18v3}.

\begin{lemma}[{\cite[Lem.~8]{ThiPinUgo18v3}}]
\label{lem:cube-face-complex} Let $F$ be a proper face in the $d$-cube $Q_{d}$. Then the antistar of $F$ is a strongly connected $(d-1)$-complex. 
\end{lemma}

Given sets $A,B,X$ of vertices in a graph $G$,  the set  $X$  {\it separates} $A$ from $B$ if every $A-B$ path in the graph contains a vertex from $X$. A set $X$ separates two vertices $a,b$ not in $X$ if it separates $\{a\}$ from $\{b\}$. We call the set $X$ a {\it separator} of the graph. 

We will also require the following three assertions. 
 
\begin{proposition}[{\cite[Prop.~1]{Ram04}}]\label{prop:cube-cutsets} Any separator $X$ of cardinality $d$ in $Q_{d}$ consists of the $d$ neighbours of some vertex in the cube and the subgraph $G(Q_{d})-X$ has exactly two components, with one of them being the vertex itself.
\end{proposition}

A set of vertices in a graph is {\it independent} if no two of its elements are adjacent. Since there are no triangles in a $d$-cube,  \cref{prop:cube-cutsets} gives at once the following corollary.
 
 \begin{corollary}\label{cor:separator-independent}  A separator of cardinality $d$ in a $d$-cube is an independent set.
 \end{corollary}

\begin{remark}\label{rmk:cubical-common-neighbours} If $x$ and $y$ are vertices of a cube, then they share at most two neighbours. In other words, the complete bipartite graph $K_{2,3}$ is not a subgraph of the cube; in fact, it is not an induced subgraph of any simple polytope \cite[Cor.~1.12(iii)]{PfePilSan12}. 
\end{remark}

\subsection{Linkedness of the $d$-cube}The linkedness of a $d$-cube was first established in \cite[Prop.~4.4]{Mes16} as part of a study of linkedness in Cartesian products of graphs. We give an alternative proof of the result. As discussed before (\cref{prop:4polytopes}), the linkedness of $Q_{4}$ is easily shown to be two.

Since we make heavy use of Menger's theorem \cite[Thm.~3.3.1]{Die05} henceforth, we remind the reader of the theorem and one of one of its consequences.

\begin{theorem}[Menger's theorem, {\cite[Sec.~3.3]{Die05}}]\label{thm:Menger} Let $G$ be a  graph, and let $A$ and $B$ be two subsets of its vertices. Then the minimum number of vertices separating $A$ from $B$ in $G$ equals the maximum number of disjoint $A-B$ paths in $G$. 
\end{theorem} 

\begin{theorem}[Consequence of Menger's theorem]\label{thm:Menger-consequence} Let $G$ be a $k$-connected graph, and let $A$ and $B$ be two subsets of its vertices, each of cardinality at least $k$. Then there are $k$ disjoint $A-B$ paths in $G$. 
\end{theorem}

Two vertex-edge paths are {\it independent} if they share no inner vertex. 

\begin{lemma} Let P be a cubical $d$-polytope with $d\ge 4$. Let $X$ be a set of $d+1$  vertices in $P$, all contained in a facet $F$. Let $k:=\floor{(d+1)/2}$. Arbitrarily label and pair $2k$ vertices in $X$ to obtain $Y:=\{\{s_{1},t_{1}\},\ldots,\{s_{k},t_{k}\}\}$. Then, for at least $k-1$ of these pairs $\{s_{i},t_{i}\}$, there is an  $X$-valid  $s_{i}-t_{i}$ path in $F$.   
\label{lem:short-distance} 
\end{lemma}
 
\begin{proof} If, for each pair in $Y$ there is an  $X$-valid path in $F$ connecting the pair, we are done. So assume there is a pair in $Y$, say $\{s_{1},t_{1}\}$, for which an $X$-valid $s_{1}-t_{1}$ path does not exist in $F$. Since $F$ is $(d-1)$-connected, there are $d-1$ independent $s_{1}-t_{1}$ paths (\cref{thm:Menger-consequence}), each containing a vertex from $X\setminus \{s_{1},t_{1}\}$; that is, the set $X\setminus \{s_{1},t_{1}\}$, with cardinality $d-1$, separates $s_{1}$ from $t_{1}$ in $F$. By \cref{prop:cube-cutsets}, the vertices in $X\setminus \{s_{1},t_{1}\}$ are the neighbours of $s_{1}$ or $t_{1}$ in $F$, say of $s_{1}$. 

Take any pair in $Y \setminus \{\{s_{1},t_{1}\}\}$, say $\{s_{2},t_{2}\}$. If there was no $X$-valid $s_{2}-t_{2}$ path in $F$, then,  by \cref{prop:cube-cutsets}, the set $X\setminus \{s_{2},t_{2}\}$ would separate $s_{2}$ from $t_{2}$ and would consist of the neighbours of $s_{2}$ or $t_{2}$ in $F$, say of $s_{2}$. But in this case, a vertex $x$ in $X\setminus \{s_{1},s_{2},t_{1},t_{2}\}$, which exists since $|X|\ge 5$, would form a triangle with $s_{1}$ and $s_{2}$, a contradiction. See also \cref{cor:separator-independent}. Since our choice of $\{s_{2},t_{2}\}$ was arbitrary, we must have an $X$-valid path in $F$ between any pair $\{s_{i},t_{i}\}$ for $i\in [2,k]$.          
\end{proof}

 For a set $Y:=\{\{s_{1},t_{1}\},\ldots,\{s_{k},t_{k}\}\}$ of pairs of vertices in a graph, a {\it $Y$-linkage} $\{L_{1},\ldots,L_{k}\}$ is a set of disjoint paths with the path $L_{i}$ joining the pair $\{s_{i},t_{i}\}$ for $i\in [1,k]$.  For a path $L:=u_{0}\ldots u_{n}$ we often write $u_{i}Lu_{j}$ for $0\le i\le j\le n$  to denote the subpath $u_{i}\ldots u_{j}$. We are now ready to prove \cref{thm:cube}.

The definition of $k$-linkedness gives the following lemma at once.

\begin{lemma}\label{lem:k-linked-def} Let $\ell\le k$. Let $X$ be a set of $2\ell$ distinct vertices of a $k$-linked graph $G$,  let $Y$ be a labelling and pairing of the vertices in $X$, and let $Z$ be a set of $2k-2\ell$ vertices in $G$ such that $X\cap Z=\emptyset$. Then there exists a $Y$-linkage in $G$ that avoids every vertex in $Z$.  
\end{lemma}
 
\begin{theorem}[{Linkedness of the cube}]\label{thm:cube}  For every $d\ne 3$, a $d$-cube is $\floor{(d+1)/2}$-linked.
\end{theorem}
\begin{proof} The cases of $d=1,2$ are trivially true. For the remaining values of $d$, we proceed by induction, with $d=4$ given by \cref{prop:4polytopes}. 
 
Let $k:=\floor{(d+1)/2}$,  then $2k-1\le d$.  Let $X$ be any set of $2k$ vertices, our terminals, in the graph of the  $d$-cube $Q_{d}$ and let $Y:=\{\{s_{1},t_{1}\},\ldots,\{s_{k},t_{k}\}\}$ be a pairing and labelling of the vertices of $X$. We aim to find a $Y$-linkage $\{L_{1},\ldots,L_{k}\}$ with $L_{i}$ joining the pair $\{s_{i},t_{i}\}$ for $i=1,\ldots,k$. For a facet $F$ of $Q_{d}$, let $F^{o}$ denote the facet opposite to $F$. 

We consider three scenarios: (1)  all the pairs in $Y$ lie in some facet of $Q_{d}$, (2)  a pair  of $Y$ lies in some facet $F$ of $Q_{d}$ but not every vertex of $X$ is in $F$, and (3) no pair of $Y$ lies in a facet of $Q_{d}$, which amounts to saying that every pair in $Y$ is at distance $d$ in $Q_{d}$. For the sake of readability, each scenario is highlighted in bold.

{\bf In the first scenario every vertex in $X$ lies in some facet $F$ of $Q_{d}$}. Hence \cref{lem:short-distance} gives an $X$-valid path $L_{1}$ in $F$ joining a pair in $Y$, say $\{s_{1},t_{1}\}$. The projection in $Q_{d}$ of every vertex in $(X\setminus \{s_{1},t_{1}\})\cap V(F)$ onto $F^o$ is not in $X$. Define $Y^{p}:=\{\{s_{2}^{p},t_{2}^{p}\},\ldots,\{s^{p}_{k},t^{p}_{k}\}\}$ as the set of $k-1$ pairs of projections of the corresponding vertices in $Y\setminus\left\{\{s_{1},t_{1}\}\right\}$ onto $F^{o}$. By the induction hypothesis on $F^o$, there is a $Y^{p}$-linkage $\{L_{2}^{p},\ldots, L_{k}^{p}\}$ with $L_{i}^{p}:=s_{i}^{p}-t_{i}^{p}$ for $i\in [2,k]$. Since $V(F^{o})$ is disjoint from $V(L_{1})\cup X$, each path $L_{i}^{p}$ can be extended with $s_{i}$ and $t_{i}$ to obtain a path $L_{i}:=s_{i}-t_{i}$ for $i\in [2,k]$. And together, all the paths $\{L_{1},\ldots, L_{k}\}$ give the desired $Y$-linkage in the cube. 

{\bf In the second scenario a pair  of $Y$, say $\{s_{1},t_{1}\}$,  lies in some facet $F$ of $Q_{d}$ but not every vertex in $X$ is in $F$.}  Let $N_{K}(x)$ denote the set of neighbours of a vertex $x$ in a face $K$ of the cube and let $N(x)$ denote the set of all the neighbours of $x$ in the cube. 

In what follows, whenever $x\in X$ we let $\{x, y\}\in Y$. Let $X_{F}:=(X\setminus \{s_{1},t_{1}\})\cap V(F)$, and partition $X_{F}$ as follows.
\begin{align*}
X_{0}&:=\{x\in X_{F}: \text{$\{x,y\}\in Y$ and $y\in N_{F}(x)$}\}\\ 
X_{1}&:=\{x\in X_{F}\setminus X_{0}: \text{$\{x,x^{p}_{F^{o}}\}\in Y$}\}\\ 
X_{2}&:=\{x\in X_{F}\setminus (X_{0}\cup X_{1}): \text{$x^{p}_{F^{o}}\not \in X$}\}\\ 
X_{3}&:=\{x\in X_{F}\setminus (X_{0}\cup X_{1}\cup X_{2}): \text{$x^{p}_{F^{o}} \in X$, and for $\{x,y\}\in Y$ and $y \in V(F^{o})$} \\ 
&\hspace{4cm}\text{there is a unique $X$-valid path $xy^{p}_{F}y$} \}\\
X_{4}&:=X_{F}\setminus(X_{0}\cup X_{1}\cup X_{2}\cup X_{3})  
\end{align*}  

Let $L_{i}:=xy$ if $x\in X_{0}\cup X_{1}$ and, $x=s_{i}$ or $x=t_{i}$, and let $L_{i}:=x\pi_{F(y)}y$ if $x\in X_{3}$ and, $x=s_{i}$ or $x=t_{i}$. 

\begin{claim}\label{claim:xvalid-path-existence} Let $X_4'\subseteq X_4$. For every vertex $x$ in $X_{2}\cup X_{3}\cup X'_{4}$, there is an $X$-valid path $M_{x}$ of length at most two from $x$ to $F^o$ such that  $(V(M_{x})\cap X)\subseteq\{x,y\}$ and the $|X_{2}\cup X_{3}\cup X'_{4}|$  paths $M_{x}$  are pairwise disjoint. 
\end{claim}

\begin{claimproof} We prove this claim by induction on the cardinality $|X'_{4}| = \ell$ of $X'_{4}$.

\textbf{In the base case $\ell=0$}, for $x\in X_{2}$, let $M_{x}=xx^{p}_{F^{o}}$. For $x\in X_{3}$, let $M_{x}$ be the unique $X$-valid path $xy^{p}_{F}y$ with $\{x,y\}\in Y$ and $y \in V(F^{o})$. It is clear that these paths are pairwise disjoint and $X$-valid.

Now suppose that the claim is true for any subset of $X_4$ of cardinality $\ell-1$. Pick a vertex $x\in X'_4$ and let $X''_4 = X'_4\setminus \{x\}$.  By the induction hypothesis, there exist $X$-valid and pairwise disjoint paths $M_z$ of length at most two from $z$ to $F^{o}$ for $z\in X_{2}\cup X_{3}\cup X''_{4}$. To prove that the claim is true for $X'_4$, we only need to construct an $X$-valid path $M_x$ disjoint from all these paths $M_z$ previously defined. We will construct it as  $xw_x\pi_{F^{o}}(w_x)$ for some $w_{x}\in N_{F(x)}$. Define $$O_{x} = \bigcup_{z\in X_{2}\cup X_{3}\cup X''_{4}} \left(M_z \cap N_F(x)\right)\; \bigcup \left(X\cap N_F(x)\right).$$
The set $O_{x}$ represents the set of  vertices in $N_{F}(x)$ that cannot be chosen as $w_{x}\in N_{F}(x)$ in the path $M_{x}$. In other words, if $N_F(x)\setminus O_{x}\ne \emptyset$ then the claim is true for $X'_4$.
 
Excluding the path $xx^{p}_{F^{o}}$, there are exactly $d-1$ disjoint paths of length two in $Q_{d}$ between $x$ and $F^{o}$, each going through an element of $N_{F}(x)$. Thus, to show that there is  a  suitable vertex $w_{x}\in N_{F}(x)$, it suffices to show an injection between $O_{x}$ and $X\setminus \{x,x^{p}_{F^{o}},y\}$, which would imply $|O_{x}|\le d-2$.  Observe that $y,x^{p}_{F^{o}}\in X\setminus O_{x}$ and $y\ne x^{p}_{F^{o}}$. 
 
For every vertex $z\in O_{x}\cap X$, map $z$ to $z$. For every $v\in O_{x}\setminus X$ with $v^{p}_{F^{o}}\in X$, map $v$ to $v^{p}_{F^{o}}$; note that $v^{p}_{F^{o}}\ne y$, since $x\not\in X_{3}$. For a vertex $w_{u}\in O_{x}\setminus X$ with $\pi_{F_{o}}(w_{u})\not \in X$ there exists a unique vertex $u\in X''_{4}\setminus O_{x}$ such that $w_{u}$ is the {\it unique} vertex in $N_{F}(x)$ on the path $M_{u}$.  Since $u\in  X_{4}$, it follows that $u^{p}_{F^{o}}\in X$. In this case, map $w_{u}$ to $u$ if $u\ne y$, otherwise map $w_{u}$ to $u^{p}_{F^{o}}$. Note that $u\not \in O_{x}$; otherwise the vertices $u$, $x$ and $w_{u}$ would all be pairwise neighbours but there are no triangles in $Q_{d}$.  See \cref{fig:Aux-Cube-Linked-Thm}(a)-(b) for a depiction of the different types of neighbours of the vertex $x\in X_{4}$ and the injection from $O_{x}$ to $X\setminus \{x,x^{p}_{F^{o}},y\}$. 
  
\begin{figure}     
\includegraphics{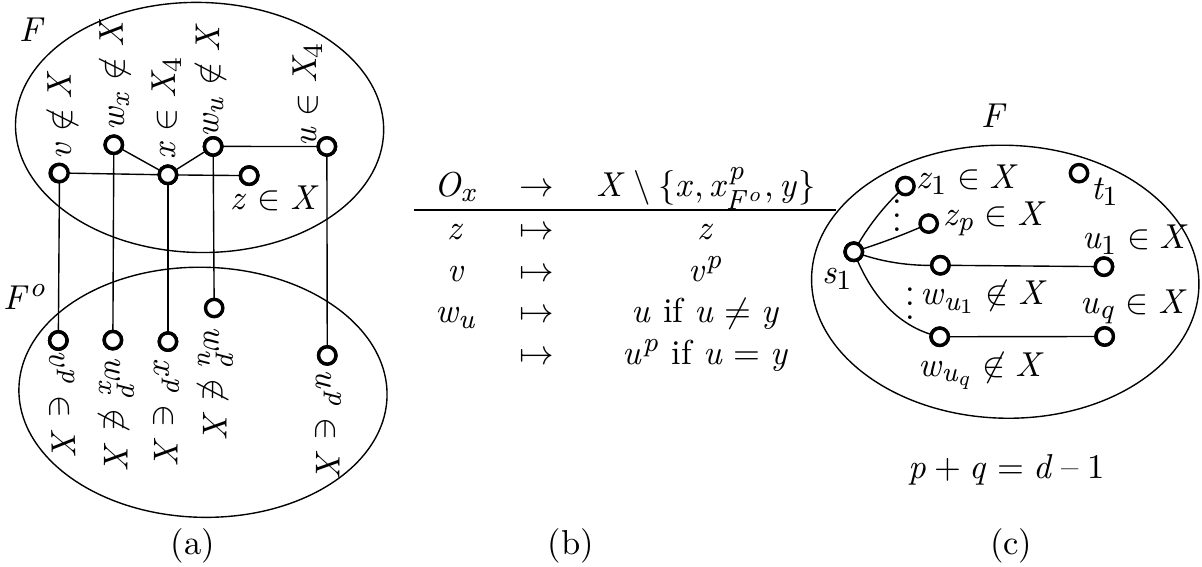}
\caption{Auxiliary figure for the second scenario \cref{thm:cube}. {\bf (a)} Types of neighbours of a vertex $x\in X_{4}$ for finding the path $M_{x}$. {\bf (b)} An injective function from $O_{x}$ to $X\setminus \{x,x^{p}_{F^{o}}y\}$. {\bf (c)} A configuration where no path $L_{1}:=s_{1}-t_{1}$ exists in $F$.}\label{fig:Aux-Cube-Linked-Thm} 
\end{figure}

The existence of an injection from $O_{x}$ to $X\setminus \{x,x^{p}_{F^{o}},y\}$ shows the existence of the vertex $w_{x}\in N_{F}(x)$, and therefore, of the desired path $M_x = xw_{x}\pi_{F^{o}}(w_x)$. This concludes the proof of the claim.\end{claimproof}
 
We now finalise this second scenario. Let $Y_{3}$ be the set of vertices $\left\{y\right\}=M_{x}\cap V(F^{o})$ for $x\in X_{3}$. Then $Y_{3}\subseteq X\cap V(F^{o})$ and $|X_{3}|=|Y_{3}|$. The pairs of terminals in $X_{0}\cup X_{1}\cup \pi_{F^{o}}(X_{1})\cup X_{3}\cup Y_{3}$ are already linked by $X$-valid paths $L_{i}$. We link the remaining pairs in $Y$ thereafter.

Applying Claim~\ref{claim:xvalid-path-existence} to $X_4$, we get the paths $M_{x}$ from all the terminals in $(X\cap V(F))\setminus (X_{0}\cup X_{1}\cup \{s_{1},t_{1}\})$ to $F^{o}$. 
For every vertex $x\in (X\cap V(F^{o})\setminus (Y_{3}\cup \pi_{F^{o}}(X_{1}))$, let $M_{x}:=x$. In this way, the paths $M_{x}$ have been defined for every vertex $x$ in $X\setminus (X_{0}\cup X_{1}\cup \pi_{F^{o}}(X_{1})\cup Y_{3}\cup \{s_{1},t_{1}\})$.  Denote by $X'$ the set of vertices in $M_{x}\cap V(F^o)$ for each $x$ in $X\setminus (X_{0}\cup X_{1}\cup \pi_{F^{o}}(X_{1})\cup X_{3}\cup Y_{3}\cup \{s_{1},t_{1}\})$.  Then

\begin{equation}\tag{*}\label{eq:cube}|X'|+|X_{1}|+|\pi_{F^{o}}(X_{1})|+|X_{3}|+|Y_{3}|\le 2(k-1)\le d-1.\end{equation} Let $Y'$ be the corresponding pairing of the vertices in $X'$: if $\{x,y\}\in Y$ with $x,y\in X\setminus (X_{0}\cup X_{1}\cup \pi_{F^{o}}(X_{1})\cup X_{3}\cup Y_{3}\cup \{s_{1},t_{1}\})$, then the corresponding pair in $Y'$ is $\{M_{x}\cap V(F^{o}),M_{y}\cap V(F^{o})\}$.  
 
The induction hypothesis ensures that $F^{o}$ is $(k-1)$-linked. As a consequence, because of \eqref{eq:cube} there is a $Y'$-linkage that avoids every vertex in $\pi_{F^{o}}(X_{1})\cup Y_{3}$ (\cref{lem:k-linked-def}). The $Y'$-linkage gives the existence of paths $L_{i}^{p}$ in $F^{o}$ between $M_{s_{i}}\cap V(F^{o})$ and $M_{t_{i}}\cap V(F^{o})$ for $s_{i},t_{i} \in X\setminus (X_{0}\cup X_{1}\cup \pi_{F^{o}}(X_{1})\cup X_{3}\cup Y_{3}\cup \{s_{1},t_{1}\})$. Each path $L_{i}^{p}$ is then extended with the paths $M_{s_{i}}$ and $M_{t_{i}}$ to obtain a path $L_{i}:=s_{i}-t_{i}$ for $s_{i},t_{i} \in X\setminus (X_{0}\cup X_{1}\cup \pi_{F^{o}}(X_{1})\cup X_{3}\cup Y_{3}\cup \{s_{1},t_{1}\})$. 
  
   It only remains to show the existence of a path $L_{1}:=s_{1}-t_{1}$ in $F$ pairwise disjoint from the paths $L_{i}$ for $i\in [2,k]$. Suppose that we cannot find a path $L_{1}$ pairwise disjoint from the other paths $L_{i}$ with  $i\in [2,k]$. Then there would be a set $S$ in $V(F)$ separating $s_{1}$ from $t_{1}$. The set $S$ would consist of terminal vertices in $X_{F}$ and nonterminal vertices on some path $M_{x}$ for $x\in X_{3}\cup X_{4}$. Each nonterminal vertex in $S$ amounts to the existence of a terminal vertex in $F^{o}$, namely $x^{p}_{F^{o}}$, since $\pi_{F^{o}}(X_{3}\cup X_{4})\subset X$. Hence, the cardinality of $S$ would be at most $|X_{F}|+|X\cap V(F^{o})|=X\setminus \{s_{1},t_{1}\}|=d-1$.  By the $(d-1)$-connectivity of $F$, the set $S$ would have cardinality $d-1$, which  would imply that every terminal in $X_{F}$ and every nonterminal in $F$ that lies on a path $M_{x}$ for $x\in X_{3}\cup X_{4}$ are in $S$. By \cref{prop:cube-cutsets}, the set $S$ would consist of the neighbours of $s_{1}$ or $t_{1}$, say of $s_{1}$. In this configuration all the vertices of $X$ would be in $F$, which is a contradiction.   Indeed, since there is no edge between any two vertices in $S$ (\cref{cor:separator-independent}), no nonterminal on a path $M_{x}$ is in $S$, and therefore, $S=X_{F}$, or equivalently, $X\subset V(F)$, as desired. The existence of the path $L_{1}$ finally settles the second  scenario. See \cref{fig:Aux-Cube-Linked-Thm}(c).
 
 It is instructive for the reader to convince themself that the proof of the second scenario works well by verifying the existence of the paths $M_{x}$ and the existence of the path $L_{1}$ for the cubes $Q_{4}$ and $Q_{5}$. 
 
{\bf Finally, let us move onto the third and final scenario: every pair in $Y$ is at distance $d$.} From \cref{lem:facets-association} it follows that there exists a pair $\{F,F^o\}$ of opposite facets of $Q_{d}$ that is not associated with $X_{s_{1}}:=X\setminus \{s_{1}\}$, since $|X\setminus \{s_{1}\}|\le d$ and there are $d$  pairs of the form $\{F,F^o\}$. This means that for every $x\in X_{s_{1}}\cap V(F)$, $\pi_{F^o}(x)\not \in X_{s_{1}}$ and that for every $x\in X_{s_{1}}\cap V(F^o)$, $\pi_{F}(x)\not \in X_{s_{1}}$.  Without loss of generality, assume that $s_{1},\ldots,s_{k}\in F^o$ and $t_{1},\ldots,t_{k}\in F$.  We can further assume that $\pi_{F^{o}}(t_{i})\ne s_{1}$ for some $t_{i}\in V(F)$ with $i\in[1,k]$, say $\pi_{F^{o}}(t_{2})\ne s_{1}$. Then  $\pi_{F^{o}}(t_{2})\not\in X$. Let $X':=\{\pi_{F}(s_{3}),\ldots,\pi_{F}(s_{k}),t_{3},\ldots,t_{k}\}$.  

By the induction hypothesis, $F$ is $(k-1)$-linked, and by \cref{lem:k-linked-def}, we can  find $k-2$ disjoint  paths $L_{i}'$ in $F$ between $\pi_{F}(s_{i})$ and $t_{i}$  for $i\in [3,k]$, with each path avoiding $\left\{t_{1},t_{2}\right\}$. Let $L_{i}:=s_{i}\pi_{F}(s_{i})L_{i}'t_{i}$. Now using the $(k-1)$-linkedness of $F^{o}$, find disjoint paths $L_{1}':=s_{1}-\pi_{F^{o}}(t_{1})$ and $L_{2}':=s_{2}-\pi_{F^{o}}(t_{2})$ in $F^{o}$, each avoiding the set $\left\{s_{3},\ldots,s_{k}\right\}$ (\cref{lem:k-linked-def});  there are  $k+2$ vertices in $\left\{\pi_{F^{o}}(t_{1}),\pi_{F^{o}}(t_{2}),s_{1},\ldots,s_{k}\right\}$ and $2k\ge k+2$ for $k\ge 3$. Let $L_{1}:=s_{1}L_{1}'\pi_{F^{o}}(t_{1})t_{1}$ and $L_{2}:=s_{2}L_{2}'\pi_{F^{o}}(t_{2})t_{2}$.

The proof of the theorem is now complete.
 \end{proof}

We are now in a position to answer Wotzlaw's question (\cite[Question 5.4.12]{Ron09}). We continue with a simple lemma from \cite[Sec.~3]{WerWot11}.
 
\begin{lemma}[{\cite[Sec.~3]{WerWot11}}]\label{lem:k-linked-subgraph} Let $G$ be a $2k$-connected graph and let $G'$ be a $k$-linked subgraph of $G$. Then $G$ is $k$-linked.
\end{lemma}
	
Now \cref{thm:cube} in conjunction with \cref{lem:k-linked-subgraph} gives the answer.

\begin{proposition} 
\label{prop:weak-linkedness-cubical} For every $d\ge 1$, a  cubical $d$-polytope is $\floor{d/2}$-linked.  
\end{proposition}
\begin{proof} Let $P$ be a cubical $d$-polytope. The results for $d=1,2$ are trivial. The case of $d=3$ follows from the connectivity of the graph of $P$, while the case of $d=4$ follows from \cref{prop:4polytopes}. For $d\ge 5$, since a facet of $P$ is a $(d-1)$-cube with $d-1\ge 4$,  by \cref{thm:cube} it is $\floor{d/2}$-linked. So \cref{lem:k-linked-subgraph} together with the $d$-connectivity of the graph of $P$ establishes the proposition.
\end{proof} 

We improve \cref{prop:weak-linkedness-cubical} in \cref{thm:cubical}. The latter theorem establishes the maximum possible linkedness of $\floor{(d+1)/2}$ for a cubical $d$-polytope with $d\ne 3$. The rest of the paper is devoted to proving \cref{thm:cubical}.

\subsection{Linkedness inside the cube}
We verify that, for every $d\ne 3$, the link of a vertex in a $(d+1)$-cube, which by \cref{prop:link-polytope} is combinatorially equivalent to a (cubical) $d$-polytope, is $\floor{(d+1)/2}$-linked  (\cref{prop:link-cubical}). In an abuse of terminology, we often think of the link as the corresponding (cubical) $d$-polytope.    
 
\begin{proposition}
\label{prop:link-cubical} For every $d\ne 3$, the link of a vertex in a $(d+1)$-cube $Q_{d+1}$ is $\floor{(d+1)/2}$-linked.  
\end{proposition}
\begin{proof} 
The proposition trivially holds for the cases of $d=1,2$, so assume $d\ge 4$. 

Let $k:=\floor{(d+1)/2}$. Let $v$ and $v^{o}$ be opposite vertices of $G(Q_{d+1})$; that is, $\dist_{Q_{d+1}}(v,v^{o})=d+1$. Let $X$ be a given set of $2k$ vertices in $\lk(v,Q_{d+1})$ and let $Y:=\{\{s_{1},t_{1}\},\ldots,\{s_{k},t_{k}\}\}$ be an arbitrary pairing of the vertices in $X$. From \cref{rmk:opposite-vertex} it follows that $\lk(v,Q_{d+1})$ is the subcomplex $Q_{d+1}- \{v,v^{o}\}$ of $Q_{d+1}$. We show that $Y$ is linked in $\lk(v,Q_{d+1})$.
 
Since $|X|-1\le d$ and there are $d+1$ pairs of opposite facets in $Q_{d+1}$, from \cref{lem:facets-association} there exists a pair $\{F,F^o\}$ of opposite facets of $Q_{d+1}$ that is not associated with $X$. This means that, for every $x\in X\cap V(F)$, its projection $\pi_{F^o}^{Q_{d+1}}(x)\not \in X$, and that, for every $x\in X\cap V(F^o)$, its projection  $\pi_{F}^{Q_{d+1}}(x)\not \in X$. Henceforth we write $\pi_{F}$ rather than $\pi_{F}^{Q_{d+1}}$. Assume that $v\in F$ and $v^{o}\in F^o$. We consider two cases based on the number of terminals in the facet $F$; for the sake of readability, the second case is in turn decomposed into two subcases highlighted in bold. 

In what follows we implicitly use the $d$-connectivity of $F$ or $F^{o}$. 
 
\begin{case} $|X\cap V(F)|=d+1$.\end{case}

Since $F$ is a $d$-cube, it is $\floor{(d+1)/2}$-linked by \cref{thm:cube}, and hence, we can find $k$ pairwise disjoint paths $L_{1},\ldots,L_{k}$ in $F$ between $s_{i}$ and $t_{i}$ for $i\in[1,k]$. If no path $L_{i}$ passes through $v$, we are done. So suppose one of those paths, say $L_{1}$, passes through $v$; there can be only one such path. If neither the projection of $s_{1}$ onto $F^{o}$ nor the projection of $t_{1}$ onto $F^{o}$ is $v^{o}$, then find a $\pi_{F^{o}}(s_{1})-\pi_{F^{o}}(t_{1})$ path  $\bar L_{1}$ in $F^{o}$ that avoids $v^{o}$. So $L_{1}$ would then become $s_{1}\pi_{F^{o}}(s_{1})\bar L_{1}\pi_{F^{o}}(t_{1})t_{1}$. If the projection of either $s_{1}$ or  $t_{1}$ onto $F^{o}$ is $v^{o}$, say that of $s_{1}$, then, since $\dist_{Q^{d+1}}(v,v^{o})= d+1\ge 5$ and $\dist(s_{1},v)= d\ge 4$, there must be a neighbour $w$ of $s_{1}$ on $L_{1}$ that is different from $v$. Find a $\pi_{F^{o}}(w)-\pi_{F^{o}}(t_{1})$  path $\bar L_{1}$ in $F^o$ that avoids $v^{o}$. So $L_{1}$ would then become $s_{1}w\pi_{F^{o}}(w)\bar L_{1}\pi_{F^{o}}(t_{1})t_{1}$.

By symmetry, the proposition also holds if  $|X\cap V(F^{o})|=d+1$. 

\begin{case} $|X\cap V(F)|\le d$.\end{case}
 
In this case, it is also true that $|X\cap V(F^{o})| \le d$. Let $X^{p}:=\pi_{F}(X)$. The set $X^{p}$ comprises the terminals in $X\cap V(F)$ together with the  projections onto $F$ of the vertices in $X\cap V(F^o)$. Then $|X^{p}|\le d+1$.
 
{\bf First suppose that some vertex in $X\cap V(F^o)$, say $t_{1}$, is adjacent to $v$: $\pi_{F}(t_{1})=v$.}   We must have that either $s_{1}\in F^{o}$ or $s_{1}\in F$.

Suppose $s_{1}\in F^{o}$.  Find an $X$-valid path $L_{1}:=s_{1}-t_{1}$ in $F^{o}$ using the $d$-connectivity of $F^{o}$ as there are at most $d$ terminals in $F^{o}$. Thanks to \cref{thm:cube}, $F$ is $k$-linked, and thus we can find $k-1$ disjoint paths $\bar L_{2},\ldots, \bar L_{k}$ between  $\pi_{F} (s_{i})$ and $\pi_{F}(t_{i})$ for $i\in[2,k]$,  all avoiding $v$ (\cref{lem:k-linked-def}). Each such path $\bar L_{i}$ extends to a path $L_{i}:=s_{i}\pi_{F}(s_{i})\bar L_{i}\pi_{F}(t_{i})t_{i}$, if necessary. So we are done in this scenario and ready to assume $s_{1}\in F$. 

Assume $s_{1}\in F$. The $k$-linkedness of $F$ ensures that in $F$ there are $k$ disjoint paths $M_{1}:=s_{1}-v$ and  $\bar L_{i}:=\pi_{F}( s_{i})-\pi_{F}(t_{i})$ for $i\in [2,k]$. As before, each path $\bar L_{i}$ ($i\in [2,k]$)  extends to a path $L_{i}:=s_{i}-t_{i}$, if necessary. If $v^{o}$ is not the projection of $s_{1}$ onto $F^{o}$, then find an $X$-valid $\pi_{F^{o}}(s_{1})-t_{1}$ path $\bar L_{1}$ in $F^o$ using the $d$-connectivity of $F^{o}$. Then $L_{1}$ would become $s_{1}\pi_{F^{o}}(s_{1})\bar L_{1}t_{1}$, and so  we are also home in this scenario. Otherwise $v^{o}$ is the projection of $s_{1}$ onto $F^{o}$, in which case $\dist_{F}(s_{1},v)=d\ge 4$. There is  a neighbour $w\in V(F)$ of $s_{1}$ on the path $M_{1}$, which is different from $v$; observe that $\pi_{F^{o}}(w)\not \in X$ since $w\not \in X^{p}$. Find an $X$-valid path $\pi_{F^{o}}(w)-t_{1}$ path $\bar L_{1}$  in $F^o$ (here use again  the $d$-connectivity of $F^{o}$). So $L_{1}$ would then become $s_{1}w\pi_{F^{o}}(w)\bar L_{1}t_{1}$.  This settles the subcase of  some vertex in  $X\cap V(F^o)$ being adjacent to $v$.  

{\bf Finally, assume no vertex in $X\cap V(F^o)$ is adjacent to $v$, and by symmetry, that no vertex in $X\cap V(F)$ is adjacent to $v^{o}$}. This subcase is handled similarly to Case 1, with the set $X^{p}\subset V(F)$ playing the role of $X$. Obtain paths $\bar L_{i}:=\pi_{F}(s_{i})-\pi_{F}(t_{i})$ in $F$ for $i\in [1,k]$, thanks to the $k$-linkedness of $F$. If one of the paths $\bar L_{i}$, say $\bar L_{1}$, passes through $v$, then obtain a new path $\bar L_{1}:=\pi_{F^{o}}(s_{1})-\pi_{F^{o}}(t_{1})$ in $F^{o}$ using its  $d$-connectivity. Each path $\bar L_{i}$ ($i\in 1,k$)  extends to a path $L_{i}:=s_{i}-t_{i}$, if necessary.  
This completes the proof of the case and of the proposition.   
\end{proof}

\cref{prop:link-cubical} fails for $d=3$ because of the possible presence of Configuration 3F (\cref{def:Conf-3F}) in the link of a vertex of the $4$-cube. For specific examples of Configuration 3F, consider \cref{fig:link-polytope} (b)-(c) and let $s_{1}, s_{2},t_{1},t_{2}$ be the vertices labelled as $1,2,3,4$, respectively.

 \subsection{Strong linkedness of the cube} 
  
With \cref{prop:4polytopes,lem:short-distance,thm:cube} at hand, it can be verified that 4-polytopes and $d$-cubes  for $d\ne 3$ enjoy a property marginally stronger than linkedness: {\it strong linkedness}.  A $d$-polytope $P$ is {\it strongly $\floor{(d+1)/2}$-linked} if its graph has at least $d+1$ vertices and, for every set $X$ of exactly $d+1$ vertices and every pairing $Y$ with $\floor{(d+1)/2}$ pairs from $X$, the set $Y$ is linked in $G(P)$ and each path joining a pair in $Y$ avoids the vertices in $X$ not being paired in $Y$. For odd $d=2k-1$ the properties of strongly $k$-linkedness and $k$-linkedness coincide, since every vertex in $X$ is paired in $Y$; but they differ for even $d=2k$. \cref{thm:4polytopes-strong-linkedness} shows that  4-polytopes are strongly 2-linked while  \cref{thm:cube-strong-linkedness} shows that $d$-cubes for $d\ne 3$ are strongly $\floor{(d+1)/2}$-linked. 
   
\begin{theorem}[Strong linkedness of 4-polytopes]\label{thm:4polytopes-strong-linkedness}   Every cubical 4-polytope is strongly 2-linked.
\end{theorem} 
\begin{proof}  

Let $G$ denote the graph of a 4-polytope $P$ embedded in $\mathbb R^{4}$. Let $X$ be a set of five vertices in $G$.  Arbitrarily pair four vertices of $X$ to obtain $Y:=\{\{s_{1},t_{1}\},\{s_{2},t_{2}\}\}$. Let $x$ be the vertex of $X$ not being paired in $Y$. We aim to find two disjoint paths $L_{1}:=s_{1}-t_{1}$ and  $L_{2}:=s_{2}-t_{2}$ such that each path $L_{i}$ avoids the vertex $x$. The proof is very similar to that of \cref{prop:3-polytopes,prop:4polytopes}. 

Consider a linear function $f$ that vanishes on a linear hyperplane $H$ passing through $\{s_{1},s_{2},t_{1},x\}$.  Assume that $f(y)>0$  for some $y\in P$ and that $f(t_{2})\ge 0$. 

Suppose first that $H$ is  a supporting hyperplane of a facet $F$ of $P$. If $t_{2}\not\in V(F)$, then find an $X$-valid $L_{1}:=s_{1}-t_{1}$ path in $F$ using the 3-connectivity of $F$. Then use \cref{lem:linear-paths}  to find an $X$-valid $s_{2}-t_{2}$ path in which each inner vertex has positive $f$-value. If instead  $t_{2}\in F$, then $X\subset V(F)$ and \cref{lem:short-distance} ensures the existence of an $X$-valid $s_{i}-t_{i}$ path in $F$ for some $i=1,2$, say for $i=1$. Then use \cref{lem:linear-paths} to find an $X$-valid $s_{2}-t_{2}$ path in which each inner vertex has positive $f$-value. So assume  $H$ intersects the interior of $P$. Then there is a vertex in $P$ with $f$-value greater than zero and a vertex with $f$-value less than zero. In this case, use \cref{lem:linear-paths} to find  an $X$-valid $s_{1}-t_{1}$ path in which each inner vertex has negative $f$-value and an $X$-valid $s_{2}-t_{2}$ path in which each inner vertex has positive $f$-value. 
\end{proof}

 Not every 4-polytope is strongly 2-linked. Take a two-fold pyramid $P$ over a quadrangle $Q$. Then $P$ is a 4-polytope on six vertices, say $s_{1},s_{2},t_{1},t_{2},x,y$. Let the sequence $s_{1},s_{2},t_{1},t_{2}$ appears in $Q$ in cyclic order, and let the vertex $x$ be in $V(P)\setminus V(Q)$. To see that $P$ is not strongly 2-linked, observe that, for every two paths $s_{1}-t_{1}$ and $s_{2}-t_{2}$ in $P$, either they intersect or one of them contains $x$.

\begin{theorem}[Strong linkedness of the cube]\label{thm:cube-strong-linkedness}  For every $d\ne 3$, a $d$-cube is strongly $\floor{(d+1)/2}$-linked. 
\end{theorem} 
	
\begin{proof} It suffices to prove the result for $d=2k$.  Let $X$ be a set of $d+1$ vertices in the $d$-cube for $d\ne 3$. Arbitrarily pair $2k$ vertices in $X$ to obtain $Y:=\{\{s_{1},t_{1}\},\ldots,\{s_{k},t_{k}\}\}$.  Let $x$ be the vertex of $X$ not being paired in $Y$. We aim to find a $Y$-linkage $\{L_{1},\ldots, L_{k}\}$ where each path $L_{i}$ joins the pair $\{s_{i},t_{i}\}$ and avoids the vertex $x$. 

The result for $d=4$ is given by \cref{thm:4polytopes-strong-linkedness}. So assume $d\ge 6$. 

From \cref{lem:facets-association} it follows that there exists a pair $\{F,F^o\}$ of opposite facets of $Q_{d}$ that is not associated with $X_{x}:=X\setminus \{x\}$, since $|X\setminus \{x\}|=d$ and there are $d$  pairs $\{F,F^o\}$ of opposite facets in $Q_{d}$. Assume $x\in V(F^{o})$.  Let $X^{p}:=\pi_{F}(X_{x})$; that is, the set $X^{p}$ comprises the vertices in $X_{x}\cap V(F)$ plus the projections of $X_{x}\cap V(F^{o})$ onto $F$. Denote by $Y^{p}$ the corresponding pairing of the vertices in $X^{p}$; that is, $Y^{p}:=\{\{\pi_{F}(s_{1}), \pi_{F}(t_{1})\},\ldots, \{\pi_{F}(s_{k}), \pi_{F}(t_{k})\}\}$. Then $|X^{p}|=d$ and $|Y^{p}|=k$.  Find a $Y^{p}$-linkage $\{L_{1}^{p},\ldots, L_{k}^{p}\}$ in $F$ with $L_{i}^{p}:=\pi_{F}(s_{i})-\pi_{F}(t_{i})$ by resorting to the $k$-linkedness of $F$ (\cref{thm:cube}).  Adding $s_{i}\in V(F^{o})$ or $t_{i}\in V(F^{o})$ to the path $L_{i}^{p}$, if necessary, we extend the linkage $\{L_{1}^{p},\ldots, L_{k}^{p}\}$ to the required $Y$-linkage. \end{proof}

\section{Connectivity of cubical polytopes}
\label{sec:cubical-connectivity}
The aim of this section is to present a couple of results related to the connectivity of strongly connected complexes in cubical polytopes.  

The first result is from  \cite{ThiPinUgo18v3}.

\begin{proposition}[{\cite[Prop.~13]{ThiPinUgo18v3}}]\label{prop:star-minus-facet}  Let $F$ be a facet in the star $\St$  of a vertex in a cubical $d$-polytope. Then the antistar of $F$ in $\St$ is a strongly connected $(d-2)$-subcomplex of $\St$. 
\end{proposition}

We proceed with two simple but useful remarks.
 
\begin{remark}
\label{rmk:opposite-vertex-1}
Let $P$ be a cubical $d$-polytope. Let  $v$  be a vertex  of $P$ and let $F$ be a face of $P$ containing $v$. In addition, let $v^{o}$ be the vertex of $F$ opposite to $v$. The smallest face in the polytope containing both $v$ and $v^{o}$ is precisely $F$.
\end{remark}

\begin{remark}
\label{rmk:two-vertices}
 For any two faces $F,J$ of a polytope, with $F$ not contained in $J$, there is a facet containing $J$ but not $F$. In particular, for any two distinct vertices of a polytope, there is a facet containing one but not the other.
\end{remark}

%Two facet-ridge paths are {\it independent} if they do not  share an inner facet.

The proof idea in \cref{prop:star-minus-facet} can be pushed a bit further to obtain a rather technical result that we prove next (\cref{lem:technical}). 

\begin{lemma} Let $P$ be a cubical $d$-polytope with $d\ge 4$. Let $s_{1}$ be any vertex in $P$ and let $\St_{1}$ be the star of $s_{1}$ in the boundary complex of $P$.  Let $s_{2}$ be any vertex in $\St_{1}$, other than $s_{1}$. Define the following sets:
\begin{itemize}
\item $F_{1}$ in $\St_{1}$, a facet containing $s_{1}$ but not $s_{2}$;
\item  $F_{12}$ in $\St_{1}$, a facet containing $s_{1}$ and $s_{2}$;
\item $\St_{12}$, the star of $s_{2}$ in $\St_{1}$ (that is, the subcomplex of $\St_{1}$ formed by the facets of $P$ in $\St_{1}$ containing $s_{2}$);
\item $\mathcal A_{{1}}$, the antistar of $F_{1}$ in $\St_{1}$; and
\item $\mathcal A_{12}$, the subcomplex of $\St_{12}$ induced by $V(\St_{12})\setminus (V(F_{1})\cup V(F_{12}))$.
\end{itemize}
Then the following assertions hold.
\begin{enumerate}[(i)]
\item The complex $\St_{12}$ is a  strongly connected $(d-1)$-subcomplex of $\St_{1}$.
\item If there are more than two facets in $\St_{12}$, then, between any two facets of $\St_{12}$ that are different from $F_{12}$, there exists a $(d-1,d-2)$-path in $\St_{12}$ that does not contain the facet $F_{12}$.

\item If $\St_{12}$ contains more than one facet, then the subcomplex $\mathcal A_{{12}}$ of $\St_{12}$ contains a spanning strongly connected $(d-3)$-subcomplex.
\end{enumerate}
\label{lem:technical}
\end{lemma} 

\begin{proof} Let us prove (i). Let $\psi$ define the natural anti-isomorphism from the face lattice of $P$ to the face lattice of its dual $P^{*}$. The facets in $\St_{1}$ correspond to the vertices in the facet $\psi(s_{1})$ in $P^{*}$ corresponding to $s_{1}$; likewise for the facets in $\st(s_{2},\B(P))$ and the vertices in $\psi(s_{2})$. The facets in $\St_{12}$ correspond to the vertices in the nonempty face $\psi(s_{1})\cap \psi(s_{2})$ of $P^{*}$. The existence a facet-ridge path in $\St_{12}$ between any two facets $J_{1}$ and $J_{2}$ of $\St_{12}$ amounts to the existence of a vertex-edge path in $\psi(s_{1})\cap \psi(s_{2})$ between $\psi(J_{1})$ and $\psi(J_{2})$. That $\St_{12}$ is a strongly connected $(d-1)$-complex now follows from the connectivity of the graph of $\psi(s_{1})\cap \psi(s_{2})$ (Balinski's theorem), as desired.

We proceed with the proof of (ii). Let $J_{1}$ and $J_{2}$ be two facets of $\St_{12}$, other than $F_{12}$. If there are more than two facets in $\St_{12}$, then the face $\psi(s_{1})\cap \psi(s_{2})$ is at least bidimensional. As a result, the graph of $\psi(s_{1})\cap \psi(s_{2})$ is at least 2-connected by Balinski's theorem. By Menger's theorem, there are at least two independent vertex-edge paths in $\psi(s_{1})\cap \psi(s_{2})$ between $\psi(J_{1})$ and $\psi(J_{2})$. Pick one such path $L^{*}$  that avoids the vertex $\psi(F_{12})$ of $\psi(s_{1})\cap \psi(s_{2})$. Dualising this path $L^{*}$ gives a $(d-1,d-2)$-path between $J_{1}$ and $J_{2}$ in  $\St_{12}$ that does not contain the facet $F_{12}$.

We finally prove (iii). Assume that $\St_{12}$ contains more than one facet. We need some additional notation.
\begin{itemize}
\item  Let $F$ be a facet in $\St_{12}$ other than $F_{12}$; it exists by our assumption on $\St_{12}$.
\item Let $\A_{1}^{F}$ denote the subcomplex  $F-V(F_{1})$; that is, $\A_{1}^{F}$ is the antistar of $F\cap F_{1}$ in $F$.  
\item Let $\A_{12}^{F}$ denote the subcomplex $F-(V(F_{1})\cup V(F_{12}))$, the subcomplex of $F$ induced by $V(F)\setminus (V(F_{1})\cup V(F_{12}))$.  
\end{itemize}
We require the following claim. 
 
\begin{claim} \label{cl:tech-lem-claim2} $\A_{12}^{F}$ contains a spanning strongly connected $(d-3)$-subcomplex $\C^{F}$. 
\end{claim}	
\begin{claimproof}  We first show that $\A_{12}^{F}\ne \emptyset$. Denoting by $s_{1}^{o}$ the vertex in $F$ opposite to $s_{1}$, we have that $s_{1}^{o}$ is not in $F_{1}$ or in $F_{12}$ by \cref{rmk:opposite-vertex-1}. So $s_{1}^{o}$  is in $\A_{12}^{F}$.  
 
Notice that $s_{1}\not \in \A_{1}^{F}$. From \cref{lem:cube-face-complex} it follows that $\A_{1}^{F}$ is a strongly connected $(d-2)$-subcomplex of $F$. Write \[\A_{1}^{F}=\C(R_{1})\cup\cdots \cup \C(R_{m}),\] where $R_{i}$ is a $(d-2)$-face of $F$ for $i\in [1,m]$. No ridge $R_{i}$ is contained in $F_{12}$; otherwise $R_{i}=F\cap F_{12}$, which implies that $s_{1}\in R_{i}$, and therefore that $s_{1}\in\A_{1}^{F}$, a contradiction. Moreover, $s_{1}^{o}\in R_{i}$ for every $i\in [1,m]$, since every ridge of $F$ contains either $s_{1}$ or $s_{1}^{o}$, and $s_{ 1}\not\in R_{i}$.

Let $\C_{i}:=\B(R_{i})-V(F_{12})$. As $R_{i}\not\subset F_{12}$, we have $\dim R_{i}\cap F_{12}\le d-3$. Hence $\C_{i}$ is nonempty.  If $R_{i}\cap F_{12}\ne \emptyset$, then  $\C_{i}$ denotes  the antistar  of $R_{i}\cap F_{12}$ in $R_{i}$, a spanning strongly connected $(d-3)$-subcomplex of $R_{i}$ by \cref{lem:cube-face-complex}. If $R_{i}\cap F_{12}= \emptyset$, then $\C_{i}$ denotes the boundary complex of $R_{i}$, again a spanning strongly connected $(d-3)$-subcomplex of $R_{i}$.

Let \[\C^{F}:=\bigcup \C_{i}.\] Then the complex $\C^{F}$ is a spanning $(d-3)$-subcomplex of $\A_{12}^{F}$; we show it is strongly connected. 

Take any two $(d-3)$-faces $W$ and $W'$ in $\C^{F}$. We find a $(d-3,d-4)$-path $L$ in $\C^{F}$ between $W$ and $W'$. There exist ridges $R$ and $R'$ in $\A_{1}^{F}$ with $W\subset R$ and $W'\subset R'$.  Since $\A_{1}^{F}$ is a strongly connected $(d-2)$-complex, there is a $(d-2,d-3)$-path $R_{i_1}\ldots R_{i_{p}}$ in $\A_{1}^{F}$ between $R_{i_1}=R$ and $R_{i_{p}}=R'$, with $R_{i_{j}}\in \A_{1}^{F}$ for $j\in [1,p]$. We will show by induction on the length $p$ of the $(d-2,d-3)$-path $R_{i_1}\ldots R_{i_{p}}$ that there is a $(d-3,d-4)$-path in $\C^F$ between $W$ and $W'$.

If $p=1$, then $R_{i_{1}} = R_{i_{p}} = R = R'$. The existence of the path follows from the strong connectivity of $\C_{i_{1}}$.

Suppose that the claim is true when the length of the path is $p-1$.  We  already established that $s_{1}^{o}\in R_{i_{j}}$ for every $j\in [1,p]$ and that $s^{o}_{1}\not \in F_{12}$. Consequently, we get that $R_{i_{p-1}}\cap R_{i_{p}}\not \subset F_{12}$, and therefore, that $\dim R_{i_{p-1}}\cap R_{i_{p}}\cap F_{12}\le d-4$. Hence  the subcomplex $\B_{i_{p-1}}:=\B(R_{i_{p-1}}\cap R_{i_{p}})-V(F_{12})$ of $\B(R_{i_{p-1}}\cap R_{i_{p}})$ is a nonempty, strongly connected $(d-4)$-complex by \cref{lem:cube-face-complex}; in particular, it contains a $(d-4)$-face $U_{i_{p}}$.  Furthermore, $\B_{i_{p-1}}\subset \C_{i_{p-1}}\cap \C_{i_{p}}$.

Let $W_{i_{p-1}}$ and $W_{i_{p}}$ be $(d-3)$-faces in $\C_{i_{p-1}}$ and $\C_{i_{p}}$ containing $U_{i_{p}}$ respectively. By the induction hypothesis, the existence of the $(d-2,d-3)$-path $R_{i_1}\ldots R_{i_{p-1}}$ implies the existence of a $(d-3,d-4)$-path $L_{p-1}$ in $\C^F$ from $W$ to $W_{i_{p-1}}$. The strong connectivity of $\C_{i_{p}}$ gives the existence of a path $L_{p}$ from $W_{i_{p}}$ to $W'$. Finally, the desired path $L$ is the concatenation of these two paths: $L=L_{p-1}L_p$. The existence of the path $L$ between $W$ and $W'$ completes the proof of \cref{cl:tech-lem-claim2}. 
\end{claimproof}

We are now ready to complete the proof of (iii). The proof goes along the lines of the proof of \cref{cl:tech-lem-claim2}. We let 
\[\St_{12}=\bigcup_{i=1}^{m} \C(J_{i}),\] where the facets $J_{1},\ldots,J_{m}$ are all the facets in $P$ containing $s_{1}$ and $s_{2}$. 

For every $i\in [1,m]$ we let  $\C^{J_{i}}$ be the spanning strongly connected $(d-3)$-subcomplex in $\A_{12}^{J_{i}}$ given by ~\cref{cl:tech-lem-claim2}. And we let \[\C:=\bigcup \C^{J_{i}}.\] Then $\C$ is a spanning $(d-3)$-subcomplex of $\A_{12}$; we show it is strongly connected. 

If there are exactly two facets in $\St_{12}$, namely $F_{12}$ and some other facet $F$, then the complex $\A_{12}$ coincides with the complex $\A_{12}^{F}$. The strong $(d-3)$-connectivity of $\A_{12}^{F}$ is then settled by \cref{cl:tech-lem-claim2}. Hence assume that there are more than two facets in  $\St_{12}$; this implies that the smallest face containing $s_{1}$ and $s_{2}$ in $\St_{12}$ is at most $(d-3)$-dimensional. 

Take any two $(d-3)$-faces $W$ and $W'$ in $\C$.  Let $J\ne F_{12}$ and $J'\ne F_{12}$ be facets of $\St_{12}$ such that $W\subset J$ and $W'\subset J'$. By (ii), we can find a $(d-1,d-2)$-path $J_{i_{1}}\ldots J_{i_{q}}$ in $\St_{12}$ between $J_{i_{1}}=J$ and $J_{i_{q}}=J'$ such that $J_{i_{j}}\ne F_{12}$ for $j\in[1,q]$. We will show that a $(d-3,d-4)$-path $L$ exists between $W$ and $W'$ in $\C$, using an induction on the length $q$ of the path $J_{i_{1}}\ldots J_{i_{q}}$.

If $q=1$, then $W$ and $W'$ belong to the same facet $F$ in $\St_{12}$, which is different from $F_{12}$. In this case, $W$ and $W'$  are both in $\A_{12}^{F}$, and consequently, \cref{cl:tech-lem-claim2} gives the desired $(d-3,d-4)$-path between $W$ and $W'$ in $\A_{12}^{F}\subseteq \C$.

Suppose that the induction hypothesis holds when the length of the path is $q-1$. 
First, we show that there exists a $(d-4)$-face $U_{q}$ in $C^{J_{i_{q-1}}}\cap C^{J_{i_{q}}}$. As $J_{i_{q-1}},J_{i_{q}}\ne F_{12}$, we obtain that $\B(J_{i_{q-1}}\cap J_{i_{q}})-V(F_{12})$ is a nonempty, strongly connected $(d-3)$-subcomplex (\cref{lem:cube-face-complex}); in particular, it contains a $(d-3)$-face $K_{q}$. We pick $U_{q}$ in $\B(K_{q})-V(F_{1})$ as follows. It holds that $K_{q}\not \subset F_{1}$; otherwise $K_{q}=J_{i_{q-1}}\cap J_{i_{q}}\cap F_{1}$, a contradiction because  $s_{1}\not \in K_{q}$  but $s_{1} \in J_{i_{q-1}}\cap J_{i_{q}}\cap F_{1}$. As a  consequence, $\B(K_{q})-V(F_{1})$ is a nonempty, strongly connected $(d-4)$-subcomplex (\cref{lem:cube-face-complex} again); in particular, it contains a desired $(d-4)$-face $U_{q}$.

Pick $(d-3)$-faces $W_{q-1}\in \C^{J_{i_{q-1}}}$ and $W_{q}\in \C^{J_{i_{q}}}$ such that both contain the $(d-4)$ face $U_{q}$. The induction hypothesis tells us that there exists a $(d-3,d-4)$-path $L_{q-1}$ from $W$ to $W_{q-1}$ in $\C$. And the strong $(d-3)$-connectivity of $\C^{J_{i_{q}}}$ ensures that there exists a $(d-3,d-4)$-path $L_q$ from $W_{q}$ to $W'$. By concatenating these two paths, we can obtain the path $L=WL_{q-1}W_{q-1}W_{q}L_{q}W'$. This completes the proof of the lemma.\end{proof}

\section{Linkedness of cubical polytopes}
\label{sec:cubical-linkedness}	 

The aim of this section is to prove that, for every $d\ne 3$, a cubical $d$-polytope is $\floor{(d+1)/2}$-linked (\cref{thm:cubical}). It suffices to prove \cref{thm:cubical} for odd $d\ge 5$; since $\floor{d/2}=\floor{(d+1)/2}$ for even $d$, \cref{prop:weak-linkedness-cubical} trivially establishes \cref{thm:cubical} in this case.

The proof of \cref{thm:cubical} heavily relies on \cref{lem:star-cubical}. To state the lemma we require a generalisation of \cref{def:Conf-3F}.

 \begin{definition}[Configuration $d$F]\label{def:Conf-dF} Let $d\ge 3$ be odd and let $X$ be a set of at least $d+1$ terminals in a cubical $d$-polytope $P$. In addition, let $Y$ be a  labelling and pairing of the vertices in $X$. A terminal of $X$, say $s_{1}$, is in {\it Configuration $d$F} if the following conditions are satisfied:
 \begin{enumerate}
 \item[(i)]  at least $d+1$ vertices of $X$ appear in a facet $F$ of $P$;
 \item[(ii)]  the terminals in the pair $\{s_{1}, t_{1}\}\in Y$ are at distance $d-1$ in $F$ (that is, $\dist_{F}(s_{1},t_{1})=d-1$); and
 \item[(iii)] the neighbours of $t_{1}$ in $F$ are all vertices of $X$. 
 \end{enumerate}
\end{definition}
 
As you may already suspect, for $d=3$, Configuration $d$F for a vertex coincides with Configuration 3F for the same vertex. 

\begin{lemma} Let $d\ge 5$ be odd and let $k:=(d+1)/2$. Let $s_{1}$ be a vertex in a cubical $d$-polytope and let $\St_{1}$  be the star of  $s_{1}$ in the polytope. Moreover, let $Y:=\{\{s_{1},t_{1}\},\ldots,\{s_{k},t_{k}\}\}$ be a labelling and pairing of $2k$ distinct vertices of $\St_{1}$.  Then the set $Y$ is linked in $\St_{1}$ if and only if the vertex $s_{1}$ is not in Configuration $d$F.  
\label{lem:star-cubical} 
\end{lemma} 
	 
We defer the proof of \cref{lem:star-cubical} to Subsection~\ref{subsec:star-cubical}. We are now ready to prove our main result, assuming \cref{lem:star-cubical}.  
 
\begin{theorem}[Linkedness of cubical polytopes]\label{thm:cubical} For every $d\ne 3$, a cubical $d$-polytope is $\floor{(d+1)/2}$-linked.
\end{theorem}
\begin{proof}

\cref{prop:weak-linkedness-cubical}  settled the case of even $d$, so we assume $d$ is odd. 

Let $d$ be odd and $d\ge 5$ and let $k:=(d+1)/2$.  Let $X$ be any set of $2k$ vertices in the graph $G$ of a cubical $d$-polytope $P$. Recall the vertices in $X$ are called terminals. Also let $Y:=\{\{s_{1},t_{1}\},\ldots,\{s_{k},t_{k}\}\}$ be a labelling and pairing of the vertices of $X$. We aim to find a $Y$-linkage $\{L_{1},\ldots,L_{k}\}$ in $G$ where $L_{i}$ joins the pair $\{s_{i},t_{i}\}$ for $i=1,\ldots,k$. Recall that a path is $X$-valid if it contains no inner vertex from $X$. 
 
The first step of the proof is to reduce the analysis space from the whole polytope to a more manageable space, the star $\St_1$ of a terminal vertex in the boundary complex of $P$, say that of $s_{1}$. We do so by considering $d=2k-1$ disjoint paths $S_{i}:=s_{i}-\St_1$ ($i\in [2,k]$) and $T_{j}:=t_{j}-\St_1$ ($j\in [1,k]$) from the terminals into $\St_{1}$. Here we resort to the $d$-connectivity of $G$. In addition, let $S_{1}:=s_{1}$. We then denote by $\bar s_{i}$ and $\bar t_{j}$ the intersection of the paths $S_{i}$ and $T_{j}$ with $\St_1$. Using the vertices  $\bar s_{i}$ and $\bar t_{i}$ for $i\in [1,k]$, define sets $\bar X$ and $\bar Y$ in $\St_{1}$, counterparts to the sets $X$ and $Y$ of $G$. In an abuse of terminology, we also say that the vertices $\bar s_{i}$ and $\bar t_{i}$ are terminals. In this way, the existence of a $\bar Y$-linkage $\{\bar L_{1},\ldots,\bar L_{k}\}$ with $\bar L_{i}:=\bar s_{i}-\bar t_{i}$ in $G(\St_{1})$ implies the existence of a $Y$-linkage $\{L_{1},\ldots,L_{k}\}$ in $G(P)$, since each path $\bar L_{i}$ ($i\in [1,k]$) can be extended with the paths $S_{i}$ and $T_{i}$ to obtain the corresponding path $L_{i}=s_{i}S_{i}\bar s_{i}\bar L_{i}\bar t_{i}T_{i}t_{i}$.

The second step of the proof is to find a $\bar Y$-linkage $\{\bar L_{1},\ldots,\bar L_{k}\}$ in $G(\St_{1})$, whenever possible.  According to \cref{lem:star-cubical}, there is a $\bar Y$-linkage in $G(\St_{1})$ provided that the vertex $s_{1}$ is not in Configuration $d$F. The existence of a $\bar Y$-linkage in turn gives the existence of a $Y$-linkage, and completes the proof of the theorem in this case.

The third and final step is to deal with Configuration $d$F for $s_{1}$. Hence assume that the vertex $s_{1}$  is in Configuration $d$F. This is implies that
\begin{enumerate}	
 \item[(i)] there exists a unique facet $F_{1}$ of $\St_1$ containing $\bar t_{1}$; that 
 \item[(ii)]   $|\bar X\cap V(F_{1})|= d+1$; and that
  \item[(iii)]  $\dist_{F_{1}}(\bar s_{1},\bar t_{1})=d-1$ and all the $d-1$ neighbours of $\bar t_{1}$ in $F_{1}$, and thus in $\St_1$, belong to $\bar X$.
\end{enumerate}	

Let $R$ be a $(d-2)$-face of $F_{1}$ containing $s_{1}^{o}=\bar t_{1}$, then $s_{1}\not \in R$. Denote by $R_{F_{1}}$ the  $(d-2)$-face of $F_{1}$ disjoint from $R$. Let $J$ be the other facet of $P$ containing $R$ and let $R_{J}$ denote the  $(d-2)$-face of $J$ disjoint from $R$. Then $R_{J}$ is disjoint from $F_{1}$. Partition the vertex set $V(R_{J})$ of $R_{J}$ into the vertex sets of two induced subgraphs $G_{\text{bad}}$ and $G_{\text{good}}$ such that $G_{\text{bad}}$ contains the neighbours of the terminals in $R$, namely $V(G_{\text{bad}})=\pi_{R_{J}}^{J}(\bar X\cap V(R))$ and $V(G_{\text{good}})=V(R_{J})\setminus V(G_{\text{bad}})$. Then $\pi_{R}^{J}(V(G_{\text{bad}}))\subseteq \bar X$ and $\pi_{R}^{J}(V(G_{\text{good}}))\cap \bar X=\emptyset$.  See \cref{fig:Aux-Linked-Thm}(a).

Consider again the paths $S_{i}$ and $T_{j}$ that bring the vertices $s_{i}$ ($i\in [2,k]$) and $t_{j}$ ($j\in [1,k]$) into $\St_1$. Also recall that the paths $S_{i}$ and $T_{j}$ intersect $\St_1$ at $\bar s_{i}$ and $\bar t_{j}$, respectively. We distinguish two cases: either at least one path $S_{i}$ or $T_{j}$ touches $R_{J}$ or no  path  $S_{i}$ or $T_{j}$ touches $R_{J}$. In the former case we redirect one aforementioned path $S_{i}$ or $T_{j}$ to break Configuration $d$F for $s_{1}$ and use \cref{lem:star-cubical}, while in the latter case we find the $\bar Y$-linkage using the antistar of $s_{1}$.  

\begin{case}\label{case:new-linkedness-thm-1} Suppose at least one path $S_{i}$ or $T_{j}$ touches $R_{J}$. \end{case}

 If possible, pick one such path, say $S_{\ell}$, for which it holds that $V(S_{\ell})\cap V(G_{\text{good}})\ne \emptyset$. Otherwise, pick one such path, say $S_{\ell}$, that does not contain $\pi_{R_{J}}^{J}(t_{1})$, if it is possible. If none of these two selections are possible, then there is exactly one path $S_{i}$ or $T_{j}$ touching $R_{J}$, say $S_{\ell}$, in  which  case  $\pi_{R_{J}}^{J}(t_{1})\in V(S_{\ell})$.

We replace the path $S_{\ell}$ by a new path $s_{\ell}-\St_{1}$ that is disjoint from the other paths $S_{i}$ and $T_{j}$ and we replace the old terminal $\bar s_{}$ by a new terminal that causes $s_{1}$ not to be in Configuration $d$F. First suppose that there exists $s_{\ell}'$  in  $V(S_{\ell})\cap V(G_{\text{good}})$. Then the old  path $S_{\ell}$ is replaced by the path $s_{\ell}S_{\ell}s'_{\ell}\pi_{R}^{J}(s_{\ell}')$, and the old terminal $\bar s_{\ell}$ is replaced by $\pi_{R}^{J}(s_{\ell}')$. Now suppose that $V(S_{\ell})\cap V(G_{\text{good}})=\emptyset$. Then every  path $S_{i}$ and $T_{j}$ that touches $R_{J}$ is disjoint from $G_{\text{good}}$. Denote by $s_{\ell}'$ the first intersection of $S_{\ell}$ with $R_J$. Let $M_{\ell}$ be a shortest path in $R_{J}$ from $s_{\ell}'\in V(G_{\text{bad}})$ to a vertex  $s_{\ell}''\in V(G_{\text{good}})$.  By our selection of $S_{\ell}$ this path $M_{\ell}$ always exists. If $s_{\ell}''\in V(G_{\text{good}})\setminus V(\St_1)$ then the old  path $S_{\ell}$ is replaced by the path $s_{\ell}S_{\ell}s'_{\ell}M_{\ell}s''_{\ell}\pi_{R}^{J}(s_{\ell}'')$, and the old terminal $\bar s_{\ell}$ is replaced by $\pi_{R}^{J}(s_{\ell}'')$. If instead  $s_{\ell}''\in V(G_{\text{good}})\cap V(\St_1)$ then the old  path $S_{\ell}$ is replaced by the path $s_{\ell}S_{\ell}s'_{\ell}M_{\ell}s''_{\ell}$, and the old terminal $\bar s_{\ell}$ is replaced by $s_{\ell}''$. Refer to \cref{fig:Aux-Linked-Thm}(b) for a depiction of this case.

In any case, the replacement of the old vertex $\bar s_{\ell}$ with the new $\bar s_{\ell}$  forces $s_{1}$ out of Configuration $d$F, and we can apply \cref{lem:star-cubical} to find a $\bar Y$-linkage.  The case of $S_{\ell}$ being equal to $T_{1}$ requires a bit more explanation in order to make sure that the vertex  $s_{1}$ does not end up in a new configuration $d$F.  Let $\A_1$ be the antistar  of $F_{1}$ in $\St_1$.  The new vertex $\bar t_{1}$ is either in $F_{1}$ or in $\A_{1}$. If the new $\bar t_{1}$ is in $F_{1}$ then it is plain that $s_{1}$ is not in Configuration $d$F. If the new vertex $\bar t_{1}$  is in $\A_{1}$,  then a new facet $F_{1}$ containing $s_{1}$ and the new $\bar t_{1}$ cannot contain all the $d-1$ neighbours of the old $\bar t_{1}$ in the old $F_{1}$, since the intersection between the new and the old $F_{1}$ is at most $(d-2)$-dimensional and no $(d-2)$-dimensional face of the old $F_{1}$ contains all the $d-1$ neighbours of the old $\bar t_{1}$. This completes the proof of the case.

\begin{case}\label{case:new-linkedness-thm-2} For any ridge $R$ of $F_{1}$ that contains $\bar t_{1}$, the aforementioned ridge $R_{J}$ in the facet $J$ is disjoint from all the paths $S_{i}$ and $T_{j}$. \end{case}
 
Consider the vertex $\bar t_{1}$ in $F_{1}$, an aforementioned ridge $R$, and the corresponding facet $J$ and ridge $R_{J}$. There is a unique neighbour of $\bar t_{1}$ in $R_{F_{1}}$, say $\bar s_{k}$, while every other neighbour of $\bar t_{1}$ in $F_{1}$ is in $R$. Let $\bar X^{p}:=\pi_{R_{J}}^{J}(\bar X\setminus\{s_{1},\bar s_{k},\bar t_{k}\})$ and let $s_{1}^{pp}:=\pi_{R_{J}}^{J}(\pi_{R}^{F_{1}}(s_{1}))$. See \cref{fig:Aux-Linked-Thm}(c). The $d-1$ vertices in $\bar X^{p}\cup\{s_{1}^{pp}\}$ can be linked in $R_{J}$ (\cref{thm:cube}) by a linkage $\{\bar L_{1}',\ldots,\bar L_{k-1}'\}$. Observe that, for the special case of $d=5$ where $R_{J}$ is a 3-cube,  the sequence $s_{1}^{pp},  \pi_{R_{J}}^{J}(\bar s_{2}), \pi_{R_{J}}^{J}(\bar t_{1}),\pi_{R_{J}}^{J}(\bar t_{2})$ cannot be in a 2-face in cyclic order, since $\dist_{R_{J}}(s_{1}^{pp},\pi_{R_{J}}^{J}(\bar t_{1}))=3$. The linkage $\{\bar L_{1}',\ldots,\bar L_{k-1}'\}$ together with the two-path $\bar L_{k}:=\bar s_{k}\pi_{R_{F_{1}}}^{F_{1}}(\bar t_{k})\bar t_{k}$ can be extended to a linkage $\{\bar L_{1},\ldots, \bar L_{k}\}$ given by  
 
 \[\bar L_{i}:=\begin{cases}
 s_{1}\pi_{R}^{F_{1}}(s_{1})s_{1}^{pp}\bar L_{1}'\pi_{R_{J}}^{J}(\bar t_{1})\bar t_{1},& \text{for $i=1$;}\\ 
 \bar s_{i}\pi_{R_{J}}^{J}(\bar s_{i})\bar L_{i}'\pi_{R_{J}}^{J}(\bar t_{i})\bar t_{i},& \text{for $i\in [2,k-1]$;}\\
 \bar s_{k}\pi_{R_{F_{1}}}^{F_{1}}(\bar t_{k})\bar t_{k},& \text{for $i=k$.}
 \end{cases}\] 

Concatenating the paths $S_{i}$ ($i\in [2,k]$) and $T_{j}$ ($j\in [1,k]$) with the linkage $\{\bar L_{1},\ldots, \bar L_{k}\}$ gives the desired $Y$-linkage. This completes the proof of the case, and with it the proof of the theorem. 
\end{proof}
	 
\begin{figure}  
\includegraphics{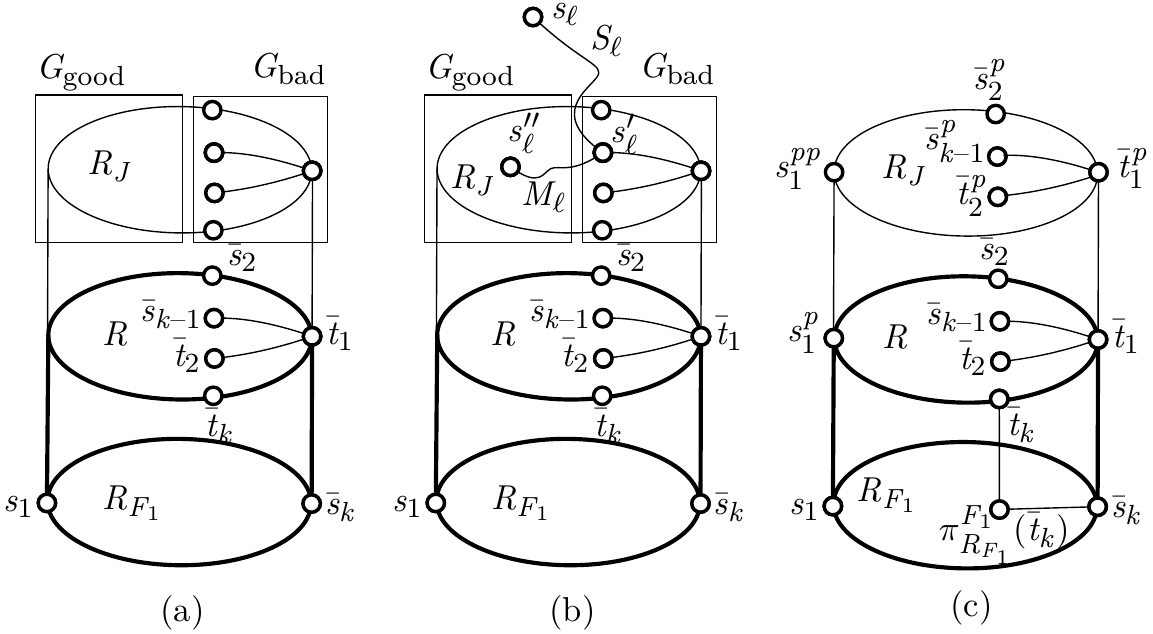}
\caption{Auxiliary figure for \cref{thm:cubical}, where the facet $F_{1}$ is highlighted in bold. {\bf (a)} A depiction of the subgraphs $G_{\text{good}}$ and $G_{\text{bad}}$ of $R_{J}$. {\bf (b)} A configuration where a path $S_{i}$ or $T_{j}$ touches $R_{J}$. {\bf (c)} A configuration where no path $S_{i}$ or $T_{j}$ touches $R_{J}$.}\label{fig:Aux-Linked-Thm} 
\end{figure}

\subsection{Proof of \cref{lem:star-cubical}}    
 \label{subsec:star-cubical}

This section is devoted to proving \cref{lem:star-cubical}.  Before starting the proof, we require a couple of results. 
   
\begin{proposition}\label{prop:star-minus-facet-linkedness} Let $F$ be a facet in the star $\St$  of a vertex in a cubical $d$-polytope. Then, for every $d\ge 2$, the antistar of $F$ in $\St$ is $\floor{(d-2)/2}$-linked. 
 \end{proposition}
 \begin{proof} Let $\St$ be the star of a vertex $s$ in a cubical $d$-polytope and let $F$ be a facet in the star $\St$. Let  $\mathcal A$ denote the antistar of $F$ in $\St$. 
 
The case of $d=2,3$ imposes no demand on $\mathcal A$, while the case $d=4,5$ amounts to establishing that the graph of $\mathcal A$ is connected. The graph of $\A$ is in fact $(d-2)$-connected, since $\mathcal A$ is a strongly connected $(d-2)$-complex (\cref{prop:star-minus-facet}). See also \cref{prop:connected-complex-connectivity}. So assume $d\ge 6$. 
 
 There is a $(d-2)$-face $R$ in  $\A$. Indeed, take a $(d-2)$-face $R'$ in $F$ containing $s$ and consider the other facet $F'$ in $\St$ containing $R'$; the  $(d-2)$-face of $F'$ disjoint from $R'$ is the desired $R$.  By \cref{thm:cube} the ridge $R$ is $\floor{(d-1)/2}$-linked but we only require it to be $\floor{(d-2)/2}$-linked. By \cref{prop:connected-complex-connectivity,prop:star-minus-facet}  the graph of $\A$ is $(d-2)$-connected.  Combining the linkedness of $R$ and the connectivity of the graph of $\A$ settles the proposition by virtue of \cref{lem:k-linked-subgraph}. 
 \end{proof}

For a pair of opposite facets $\{F,F^{o}\}$ in a cube, the restriction of the projection $\pi_{F^{o}}:Q_{d}\to F^{o}$ (\cref{def:projection}) to $F$ is a bijection from $V(F)$ to $V(F^{o})$. With the help of $\pi$, given the star $\St$ of a vertex $s$ in a cubical polytope and a facet $F$ in  $\St$, we can define an injection from the vertices in $F$,  except  the vertex opposite to $s$, to the antistar of $F$ in $\St$. Defining this injection is the purpose of \cref{lem:projections-star}.

 \begin{lemma}\label{lem:projections-star} Let $F$ be a facet in the star $\St$  of a vertex $s$ in a cubical $d$-polytope. Then there is an injective function, defined on the vertices of $F$ except the vertex $s^{o}$ opposite to $s$, that maps each such vertex in $F$ to a neighbour in $V(\St)\setminus V(F)$.  
 \end{lemma}  
\begin{proof} We construct the aforementioned injection $f$ between $V(F)\setminus \{s^{o}\}$ and $V(\St)\setminus V(F)$ as follows. Let $R_{1},\ldots, R_{d-1}$  be the $(d-2)$-faces of $F$ containing $s$, and let $J_{1},\ldots,J_{d-1}$ be the other facets of $\St$ containing $R_{1},\ldots, R_{d-1}$, respectively. Every vertex in $F$ other than $s^{o}$ lies in $R_{1}\cup\cdots\cup R_{d-1}$. Let $R_{i}^{o}$ be the $(d-2)$-face in $J_{i}$ that is opposite to $R_{i}$ for $i\in[1,d-1]$. For every vertex $v$ in $V(R_{j})\setminus (V(R_{1})\cup \cdots \cup V(R_{j-1}))$ define  $f(v)$ as the projection $\pi$ in $J_{j}$ of $v$ onto $V(R_{j}^{o})$, namely $f(v):=\pi_{R^{o}_j}(v)$; observe that  $\pi_{R^{o}_j}(v)\in V(R_{j}^{o})\setminus (V(R_{1}^{o})\cup \cdots \cup V(R_{j-1}^{o}))$. Here $R_{-1}$ and $R_{-1}^{o}$ are empty sets. The function $f$ is well defined as $R_{i}$ and $R_{i}^{o}$ are opposite  $(d-2)$-cubes in the $(d-1)$-cube $J_{i}$.

To see that $f$ is an injection, take distinct vertices $v_{1},v_{2}\in V(F)\setminus \{s^{o}\}$, where $v_{1}\in V(R_{i})\setminus (V(R_{1})\cup \cdots \cup V(R_{i-1}))$ and $v_{2}\in V(R_{j})\setminus (V(R_{1})\cup \cdots \cup V(R_{j-1}))$ for $i\le j$. If $i=j$ then $f(v_{1})=\pi_{R^{o}_{i}}(v_{1})\ne \pi_{R^{o}_{i}}(v_{2})=f(v_{2})$. If instead $i< j$ then $f(v_{1})\in V(R_{i}^{o})\subseteq  V(R_{1}^{o})\cup \cdots \cup V(R_{j-1}^{o})$, while $f(v_{2})\not \in V(R_{1}^{o})\cup \cdots \cup V(R_{j-1}^{o})$. \end{proof}

\begin{proof}[Proof of \cref{lem:star-cubical}]

Let $d\ge 5$ be odd and let $k:=(d+1)/2$.  Let $s_{1}$ be a vertex in a cubical $d$-polytope $P$ and let $\St_{1}$ denote the star of $s_{1}$ in $\B(P)$. Let $X$ be any set of $2k$ vertices in the graph $G(\St_{1})$ of  $\St_{1}$. The vertices in $X$ are our terminals. Also let $Y:=\{\{s_{1},t_{1}\},\ldots,\{s_{k},t_{k}\}\}$ be a labelling and pairing of the vertices of $X$. We aim to find a $Y$-linkage $\{L_{1},\ldots,L_{k}\}$ in $G$ where $L_{i}$ joins the pair $\{s_{i},t_{i}\}$ for $i=1,\ldots,k$. Recall that a path is $X$-valid if it contains no inner vertex from $X$.

We consider a facet $F_{1}$ of $\St_1$ containing $t_{1}$ and having the largest possible number of terminals.

The necessary condition of $Y$ being linked in $\St_{1}$ is  easy to prove. Suppose that the vertex $s_{1}$ is in Configuration $d$F. Since $\dist_{F_{1}}(s_{1},t_{1})=d-1$, it follows that $F_{1}$ is the only facet of $\St_{1}$ that contains $t_{1}$. Then all the neighbours of $t_{1}$ in $F_{1}$, and thus, in $\St_{1}$ are in $X$. As a consequence, every $s_{1}-t_{1}$ path in $\St_{1}$ must touch $X$. Hence $Y$ is not linked.  

We decompose the sufficiency proof into four cases based on the number of terminals in $F_{1}$, proceeding from the more manageable case to the more involved one. 
\begin{enumerate}

\item[\cref{case:new-linkedness-thm-1}.] $|X\cap V(F_{1})|=d$.
\item[\cref{case:new-linkedness-thm-2}.] $3\le |X\cap V(F_{1})|\le d-1$.
\item[\cref{case:new-linkedness-thm-3}.] $|X\cap V(F_{1})|=2$ .
\item[\cref{case:new-linkedness-thm-4}.] $|X\cap V(F_{1})|=d+1$ and the vertex $s_{1}$ is not in Configuration $d$F.  
 \end{enumerate}  

The sufficiency proof of \cref{lem:star-cubical} is long, so we outline the main ideas. We let $\A_1$ be the antistar of $F_{1}$ in $\St_1$ and let $\Lk_1$ be the link of $s_{1}$ in $F_{1}$. Using the $(k-1)$-linkedness of $F_{1}$ (\cref{thm:cube}), we link as many pairs of terminals in $F_{1}$ as possible through  disjoint $X$-valid paths $L_{i}:=s_{i}-t_{i}$. For those terminals that cannot be linked in $F_{1}$, if possible we use  the injection from $V(F_{1})$ to $V(\A_{1})$ granted by \cref{lem:projections-star} to find a set $N_{\A_{1}}$ of pairwise distinct neighbours in $\A_{1}$ not in $ X$. Then, using the $(k-2)$-linkedness of $\A_{1}$ (\cref{prop:star-minus-facet-linkedness}), we link the corresponding pairs of terminals in $\A_{1}$ and vertices in $N_{\A_{1}}$  accordingly.  This general scheme does not always work, as the vertex $s_{1}^{o}$ opposite to $s_{1}$ in $F_{1}$ does not have an image in $\A_{1}$ under the aforementioned injection or the image of a vertex in $F_{1}$ under the injection may be a terminal. In those scenarios we resort to ad hoc methods, including linking corresponding pairs in the link of $s_{1}$ in $F_{1}$, which is $(k-1)$-linked by \cref{prop:link-cubical} and does not contain $s_{1}$ or $s_{1}^{o}$, or linking corresponding pairs in ridges disjoint from $F_{1}$, which are $(k-1)$-linked by \cref{thm:cube}.
 
To aid the reader, each case is broken down into subcases highlighted in bold.  

Recall that, given a pair $\{F,F^{o}\}$ of opposite facets in a cube $Q$,  for every vertex $z\in V (F)$ we denote by $z^{p}_{F^o}$ or  $\pi_{F^o}^{Q}(z)$ the unique neighbour of $z$ in $F^{o}$.

\begin{case}\label{case:new-linkedness-thm-1}  $|X\cap V(F_{1})|= d$.\end{case}
  
Without loss of generality,  assume that $  t_{2}\not \in V(F_{1})$. 

{\bf Suppose first that $\dist_{F_{1}} (  s_{2},s_{1})<d-1$}.  There exists a neighbour $  s_{2}'$ of $  s_{2}$ in $\A_1$. With the use of the strong $(k-1)$-linkedness of $F_{1}$ (\cref{thm:cube-strong-linkedness}), find disjoint paths $  L_{1}:=s_{1}-  t_{1}$ and $  L_{i}:=  s_{i}-  t_{i}$ ($i\in [3,k]$) in $F_{1}$, each avoiding $  s_{2}$. Find a path $  L_{2}$ in $\St_{1}$ between $  s_{2}$ and $  t_{2}$ that consists of the edge $  s_{2}  s_{2}'$ and a subpath in $\A_1$ between $  s_{2}'$ and $  t_{2}$, using the connectivity of $\A_1$ (see~\cref{prop:star-minus-facet}). The paths $  L_{i}$ ($i\in [1,k]$) give the desired $Y$-linkage. 

{\bf Now assume $\dist_{F_{1}}(  s_{2},s_{1})=d-1$}. Since $2k-1=d$ and there are $d-1$ pairs of opposite $(d-2)$-faces in $F_{1}$, by  \cref{lem:facets-association} there exists a pair $\{R,R^o\}$ of opposite ridges of $F_{1}$ that is not associated with the set $  X_{s_{2}}:=(  X\cap V(F_{1}))\setminus \{  s_{2}\}$, whose cardinality is $d-1$. Assume $  s_{2}\in R$. Then $s_{1}\in R^{o}$.

Suppose all the neighbours of $  s_{2}$ in $R$ are in $   X$; that is, $N_{R}(  s_{2})=  X\setminus\{s_{1},  s_{2},  t_{2}\}$. The projection $\pi_{R^{o}}^{F_{1}}(  s_{2})$ of $  s_{2}$ onto $R^{o}$ is not in $  X$ since $s_{1}$ is the only terminal in $R^{o}$ and $\dist_{F_{1}}(  s_{2},s_{1})=d-1\ge 2$. Next find disjoint paths $  L_{i}:=  s_{i}-  t_{i}$ for $i\in [3,k]$ in $R$ that do not touch $  s_{2}$ or $  t_{1}$, using the $(k-1)$-linkedness of $R$ if $d\ge 7$ (\cref{lem:k-linked-def}) or the 3-connectivity of $R$ if $d=5$. With the help of \cref{lem:projections-star}, find a  neighbour   $s_{2}'$ of $\pi_{R^{o}}^{F_{1}}(  s_{2})$ in $\A_1$, and with the connectivity of $\A_1$, a path $  L_{2}$ between $  s_{2}$ and $  t_{2}$  that consists of the length-two path $  s_{2}\pi_{R^{o}}^{F_{1}}(  s_{2})s_{2}'$ and a subpath in $\A_1$ between  $s_{2}'$ and $  t_{2}$. Finally, find a path $  L_{1}$ in $F_{1}$ between $s_{1}$ and $  t_{1}$ that consists of the edge $  t_{1}\pi_{R^{o}}^{F_{1}}(  t_{1})$ and  a subpath in $R^{o}$ disjoint from $\pi_{R^{o}}^{F_{1}}(  s_{2})$  (here use the 2-connectivity of $R^{o}$).  The paths $  L_{i}$ ($i\in [1,k]$)  give the desired $Y$-linkage.
 
Thus assume there exists a neighbour $\bar  s_{2}$ of $  s_{2}$ in $V(R)\setminus   X$.  Let $ X_{R^{o}}:=\pi_{R^{o}}^{F_{1}}( X\setminus \{ s_{2}, t_{2}\})$. Find a path $ L_{2}$ between $ t_{2}$ and $ s_{2}$ that consists of the edge $ s_{2} \bar  s_{2}$ and a subpath in $\A_1$ between $ t_{2}$ and a neighbour $ s_{2}'$ of $\bar  s_{2}$ in $\A_1$. 

Let $d\ge 7$. Find disjoint paths $ L_{i}:=\pi_{R^{o}}^{F_{1}}( s_{i})-\pi_{R^{o}}^{F_{1}}( t_{i})$ ($i\in [1,k]$ and $i\ne 2$) in $R^{o}$ linking the $d-1$ vertices in $ X_{R^{o}}$ using the $(k-1)$-linkedness of $R^{o}$; add the  edge $\pi_{R^{o}}^{F_{1}}( t_{i}) t_{i}$ to $ L_{i}$ if $ t_{i}\in R$ or the  edge $\pi_{R^{o}}^{F_{1}}( s_{i}) s_{i}$ to $ L_{i}$ if $ s_{i}\in R$. The disjoint paths $ L_{i}$ ($i\in [1,k]$) gives the desired $Y$-linkage. 
 
Let $d=5$. If the sequence $s_{1},\pi_{R^{o}}^{F_{1}}( s_{3}), \pi_{R^{o}}^{F_{1}}( t_{1}),\pi_{R^{o}}^{F_{1}}( t_{3})$ in $ X_{R^{o}}$ is not in a 2-face of $R^{o}$ in cyclic order, then the same reasoning as in the case of $d\ge 7$ applies. Thus assume otherwise. This in turn implies that $\pi_{R}^{F_{1}}( s_{3})\not\in\{ s_{2}, s_{2}'\}$ and $\pi_{R}^{F_{1}}( t_{3})\not\in\{ s_{2}, s_{2}'\}$, since $\dist_{F_{1}}(s_{1}, s_{2})=4$.  

Find a path $ L_{3}'$ in $R$  between $\pi_{R}^{F_{1}}( s_{3})$ and $\pi_{R}^{F_{1}}( t_{3})$ such that $ L_{3}'$ is disjoint from both $ s_{2}$ and $ s_{2}'$ and disjoint from $ t_{1}$ if $ t_{1}\in R$; here  use \cref{cor:separator-independent}, which ensures that the vertices $ s_{2}$, $ s_{2}'$ and $ t_{1}$, if they are all in $R$, cannot separate $\pi_{R}^{F_{1}}( s_{3})$ from $\pi_{R}^{F_{1}}( t_{3})$ in $R$,  since a separator of size three in $R$ must be an independent set. Extend the path $ L_{3}'$ in $R$ to a path $ L_{3}:= s_{3}\pi_{R}^{F_{1}}(s_{3})L_{3}'\pi_{R}^{F_{1}}(t_{3})t_{3}$ in $F_{1}$, if necessary. Find a path $ L_{1}':=s_{1}-\pi_{R^{o}}^{F_{1}}( t_{1})$ in $R^{o}$ disjoint from $\pi_{R^{o}}^{F_{1}}(s_{3})$ and $\pi_{R^{o}}^{F_{1}}(t_{3})$, using the 3-connectivity of $R^{o}$. Extend $L_{1}'$ to a path $L_{1}:=s_{1}L_{1}'\pi_{R^{o}}^{F_{1}}(t_{1})t_{1}$ in $F_{1}$, if necessary. The linkage $\{ L_{1}, L_{2},  L_{3}\}$  is a $Y$-linkage.  This completes the proof of \cref{case:new-linkedness-thm-1}.

\begin{case}\label{case:new-linkedness-thm-2}  $3\le | X\cap V(F_{1})|\le d-1$.\end{case}
 Since $2k-1=d$ and there are $d-1$ pairs of opposite facets in $F_{1}$, by \cref{lem:facets-association} there exists a pair $\{R,R^o\}$ of opposite ridges of $F_{1}$ that is not associated with $ X\cap V(F_{1})$. Assume $s_{1}\in R$. We consider two subcases according to whether $ t_{1}\in R$ or $ t_{1}\in R^{o}$.

{\bf Suppose first that $ t_{1}\in R$}.  The $(d-2)$-connectivity of $R$ ensures the existence of an $X$-valid path $ L_{1}:=s_{1}- t_{1}$ in $R$. Let 
\[ X_{R^o}:=\pi_{R^{o}}^{F_{1}}(( X\setminus \{s_{1}, t_{1}\})\cap V(F_{1})).\] Then  $1\le | X_{R^o}|\le d-3$. Let $s_{1}^{o}$ be the vertex opposite to $s_{1}$ in $F_{1}$; the vertex $s_{1}^{o}$ has no neighbour in $\A_{1}$. 

Let $\bar Z$ be a set of $|V(\A_{1})\cap  X|$ distinct vertices in $V(R^o)\setminus ( X_{R^o}\cup \{s_{1}^{o}\})$. Use \cref{lem:projections-star} to obtain a set $Z$ in $\A_1$ of $|\bar Z|$ distinct vertices  adjacent to vertices in $\bar Z$. Then $|Z|=|V(\A_{1})\cap  X|\le d-2$. To see that $|\bar Z|\le |V(R^o)\setminus ( X_{R^o}\cup \{s_{1}^{o}\})|$, observe that, for $d\ge 5$ and   $| X_{R^o}|\le d-3$, we get \[|V(R^o)\setminus ( X_{R^o}\cup \{s_{1}^{o}\})|\ge 2^{d-2}-(d-3)-1\ge d-2\ge |\bar Z|=|Z|.\]

Using the $(d-2)$-connectivity of $\A_1$ (\cref{prop:star-minus-facet}) and Menger's theorem, find disjoint paths $\bar S_{i}$ and $\bar T_{j}$ ($i,j\ne1$) in $\A_{1}$ between $V(\A_{1})\cap  X$ and $Z$. Then produce disjoint paths $S_{i}$ and $T_{j}$ ($i,j\ne1$) from terminals $s_{i}$ and $t_{j}$ in $\A_1$, respectively,  to $R^o$ by adding edges $z_{\ell}\bar z_{\ell}$ with $z_{\ell}\in Z$ and $\bar z_{\ell}\in \bar Z$ to the corresponding paths $\bar S_{i}$ and $\bar T_{j}$. If $ s_{i}$ or $ t_{j}$ is already in $R^o$, let $S_{i}:= s_{i}$ or $ T_{j}:= t_{j}$, accordingly. If instead $ s_{i}$ or $ t_{j}$ is in $R$, let $ S_{i}$ be the edge $ s_{i}\pi_{R^o}^{F_{1}}( s_{i})$ or let $ T_{j}$ be the edge $ t_{j}\pi_{R^o}^{F_{1}}( t_{j})$.  It follows that the paths $S_{i}$ and $ T_{i}$ for $i\in [2,k]$ are all pairwise disjoint. Let $ X_{R^o}^{+}$ be the intersections of $R^o$ and the paths $ S_{i}$ and $ T_{j}$ ($i,j\ne 1$). Then $| X_{R^o}^{+}|=d-1$.  Suppose that $X^{+}_{R^{o}}=\left\{\bar s_{2},\bar t_{2},\ldots, \bar s_{k},\bar t_{k}\right\}$. The corresponding pairing  $ Y_{R^o}^{+}$ of the vertices in $X_{R^o}^{+}$ can be linked through paths $\bar L_{i}:=\bar s_{i}-\bar t_{i}$ ($i\in [2,k]$) in $R^{o}$ using the $(k-1)$-linkedness of $R^o$ (\cref{thm:cube}). See \cref{fig:Aux-Linked-Thm-NewCase2}(a) for a depiction of this configuration. In this case, the desired $Y$-linkage is given by the following paths.
 
 \[L_{i}:=\begin{cases}s_{1} L_{1} t_{1},& \text{for $i=1$;}\\s_{i}S_{i}\bar s_{i}\bar L_{i}\bar t_{i} T_{i} t_{i},& \text{otherwise.}
 \end{cases}\]   
    
\begin{figure}
\includegraphics{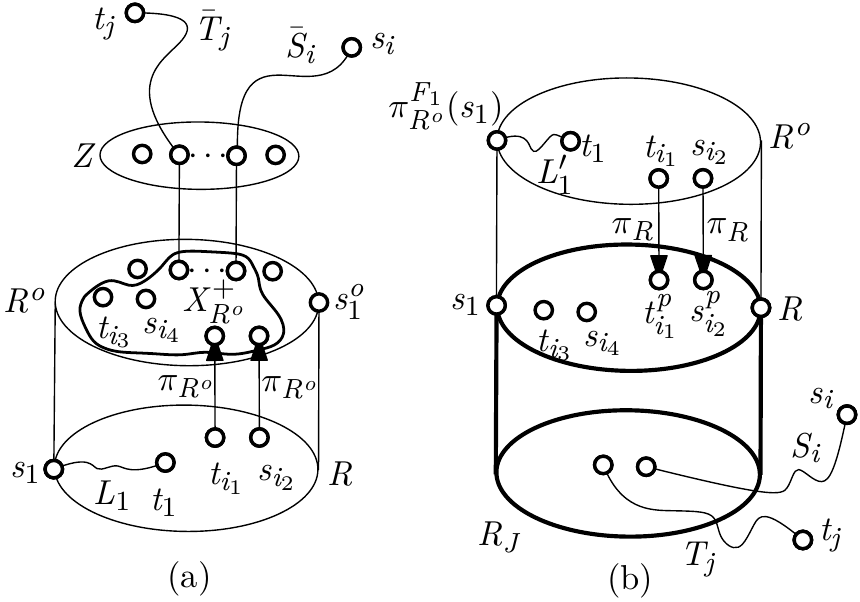}
\caption{Auxiliary figure for \cref{case:new-linkedness-thm-2} of \cref{lem:star-cubical}. {\bf (a)} A configuration where $ t_{1}\in R$  and the subset $ X^{+}_{R^{o}}$ of $R^{o}$ is highlighted in bold. {\bf (b)} A configuration where $ t_{1}\in R^{o}$  and the facet $J$ is highlighted in bold. }\label{fig:Aux-Linked-Thm-NewCase2} 
\end{figure}

Some comments for $d=5$ are in order. By virtue of \cref{prop:3-polytopes}, we need to make sure that the sequence $\bar s_{2},\bar s_{3},\bar t_{2},\bar t_{3}$ in $ X^{+}_{R^{o}}$ is not in a 2-face of $R^{o}$ in cyclic order.  To ensure this, we need to be a bit more careful when selecting the vertices in $\bar Z$. Indeed, if there are already two vertices in $ X_{R^{o}}$ at distance three in $R^{o}$, no care is needed when selecting $\bar Z$, so proceed as in the case of $d\ge 7$. Otherwise, pick a vertex $\bar z \in \bar Z\subseteq V(R^{o})\setminus ( X_{R^{o}}\cup \{s_{1}^{o}\})$ such that $\bar z$ is the unique vertex in $R^{o}$ with $\dist_{R^{o}}(\bar z,x)=3$ for some vertex $x\in  X_{R^{o}}$; this vertex $x$ exists because $| X\cap V(F_{1})|\ge 3$. Selecting such a $\bar z\ne s_{1}^{o}$ is always possible because $s_{1}^{o}$ is not at distance three in $R^{o}$ from {\it any} vertex in $ X_{R^{o}}$: the unique vertex in $R^{o}$ at distance three from $s_{1}^{o}$ is $\pi_{R^{o}}^{F_{1}}(s_{1})$, and $\pi_{R^{o}}^{F_{1}}(s_{1})\not\in  X$ because the pair $\{R,R^{o}\}$ is not associated with $ X\cap V(F_{1})$. Once $\bar z$ is selected, the set $Z$ will contain a neighbour $z$ of $\bar z$. In this way,  some path $ S_{i}$ or $ T_{j}$ bringing terminals $ s_{i}$ or $ t_{j}$ in $\A_{1}$  into $R^{o}$ through $Z$ would use the vertex $z$, thereby ensuring that $x$ and $\bar z$ would be both in  $X^{+}_{R^{o}}$. This will cause the   the sequence $\bar s_{2},\bar s_{3},\bar t_{2},\bar t_{3}$ not to be in a 2-face, and thus, not in cyclic order.  
   
{\bf Suppose now that $ t_{1}\in R^o$}.  Let \[ X_{R}:=\pi_{R}^{F_{1}}(( X\setminus\{ t_{1}\})\cap V(F_{1})).\] There are at most $d-2$ terminal vertices in $R^o$. Therefore, the $(d-2)$-connectivity of $R^o$ ensures the existence of an $ X$-valid $\pi_{R^{o}}^{F_{1}}(s_{1})- t_{1}$ path $\bar L_{1}$ in $R^o$. Then let $ L_{1}:=s_{1}\pi_{R^{o}}^{F_{1}}(s_{1})\bar L_{1} t_{1}$.  Let  $J$ be the other facet in $\St_1$ containing $R$ and let $R_{J}$ be the $(d-2)$-face of $J$ disjoint from $R$. Then $R_{J}\subset \A_{1}$. Since there are at most $d-2$ terminals in $\A_1$ and since $\A_1$ is $(d-2)$-connected (\cref{prop:star-minus-facet}), we can find corresponding disjoint paths $ S_{i}$ and $ T_{j}$ bringing the terminals in $\A_{1}$ to $R_{J}$ (\cref{thm:Menger-consequence}).  For terminals $ s_{i}$ and $ t_{j}$ in $ X\cap V(R)$, let $ S_{i}:= s_{i}$ and  $ T_{j}:= t_{j}$ for $i,j\ne 1$, while for terminals $ s_{i}$ and $ t_{j}$ in $ X\cap V(R^o)$, let $ S_{i}:= s_{i}\pi_{R}^{F_{1}}( s_{i})$ and $ T_{j}:= t_{j}\pi_{R}^{F_{1}}( t_{j})$ for $i,j\ne 1$. Let $ X_{J}$ be the set of the intersections of the paths $S_{i}$ and $ T_{j}$ with $J$ plus the vertex $s_{1}$. Then $ X_{J}\subset V(J)$ and $| X_{J}|=d$ (since $ t_{1}\ \in R^{o}$). Suppose that $X_{J}=\left\{s_{1},\bar s_{2},\bar t_{2},\ldots, \bar s_{k},\bar t_{k}\right\}$ and let $Y_{J}=\left\{\left\{\bar s_{2},\bar t_{2}\right\},\ldots, \left\{\bar s_{k},\bar t_{k}\right\}\right\}$ be a pairing of $X_{J}\setminus\left\{s_{1}\right\}$. 

Resorting to the strong $(k-1)$-linkedness of the facet $J$ (\cref{thm:cube-strong-linkedness}), we obtain $k-1$ disjoint paths $\bar L_{i}:=\bar s_{i}-\bar t_{i}$ for $i\ne 1$ that correspondingly link $Y_{J}$ in $J$, with all the paths avoiding $s_{1}$. See \cref{fig:Aux-Linked-Thm-NewCase2}(b) for a depiction of this configuration. In this case, the desired $Y$-linkage is given by the following paths.
\[L_{i}:=\begin{cases}s_{1}L_{1} t_{1},& \text{for $i=1$;}\\s_{i}S_{i} \bar L_{i} T_{i} t_{i},& \text{otherwise.}
 \end{cases}\] 	  
  
\begin{case}\label{case:new-linkedness-thm-3}  $| X\cap V(F_{1})|=2$.\end{case}
  
In this case, we have that $|V(\A_1)\cap  X|=d-1$. The proof of this case requires the definition of several sets. For quick reference and ease of readability, we place most of these definitions in itemised lists. We begin with the following sets:

\begin{itemize}
\item $\St_{12}$, the star of $ s_{2}$ in $\St_{1}$ (that is, the complex formed by the facets of $P$ containing $s_{1}$ and $ s_{2}$);
\item $G(\St_{12})$, the graph of $\St_{12}$; and 
\item $\Gamma_{12}$, the subgraph of $G(\St_{12})$ and $G(\A_1)$ that is induced by  $V(\St_{12})\setminus V(F_{1})$.
\end{itemize}
It follows that every neighbour in $G(\A_{1})$ of $s_{2}$ is in $\Gamma_{12}$; in other words, the set of neighbours of $s_{2}$ in each subgraph is the same:
\begin{equation}\label{eq:neighbourhood}
	N_{\Gamma_{12}}(s_{2})=N_{G(\A_{1})}(s_{2}). 
\end{equation}

{\bf The first step for this case is to bring the terminals in  $\A_{1}$ into $\Gamma_{12}$.} Denote by $ S_{i}$ an $ X$-valid path in $\A_1$ from the terminal $ s_{i}\in \A_1$ to  $\Gamma_{12}$. Let $V(S_{i})\cap V(\Gamma_{12})=\left\{\hat s_{i}\right\}$. Similarly, define $ T_{j}$ and $\hat t_{j}$.  The existence of these $d-2$ pairwise disjoint $ X$-valid paths $ S_{i}$ and $ T_{j}$ is ensured  by the $(d-2)$-connectivity of the graph $G(\A_1)$ of $\A_1$, which in turn is guaranteed by \cref{prop:star-minus-facet}.  By \eqref{eq:neighbourhood} each path $ S_{i}$ or $ T_{j}$ touches $\Gamma_{12}$ at a vertex other than $ s_{2}$; this is so because each such path will need to reach the neighbourhood of $s_{2}$ in $\Gamma_{12}$ before reaching $s_{2}$.  Every terminal vertex $ x$ already in $\Gamma_{12}$ is also denoted by $\hat x$, and the corresponding path $ S_{i}$ or $ T_{j}$ consists only of the vertex $\hat x$. We also let $\hat s_{2}$ denote $ s_{2}$. The set of vertices $\hat x$ is accordingly denoted by $\hat X$. Then $|\hat X|=d-1$. Abusing terminology, since there is no potential for confusion, we  call the vertices in $\hat X$ terminals as well. \Cref{fig:Aux-Linked-Thm-Case4}(a) depicts this configuration. 
   
Pick a facet   
\begin{itemize} 
\item $F_{12}$ in $\St_{12}$ that contains $\hat t_{2}$.
\end{itemize}
An important point is that $ t_{1}$ is not in $F_{12}$; otherwise $F_{12}$ would contain $s_{1}$,$s_{2}$ and  $t_{1}$, and it should have been chosen instead of $F_{1}$.   
 
{\bf The second step is to find a path $L_{1}$ in $F_{1}$ between $s_{1}$ and $ t_{1}$ such that $V(L_{1})\cap V(F_{12})=\left\{s_{1}\right\}$.}
 
 To see the existence of such a path, note that the intersection of $F_{12}$ and $F_{1}$ is at most a $(d-2)$-face containing $s_{1}$ (but not $t_{1}$), which is contained in a $(d-2)$-face $R$  of $F_{1}$ containing $s_{1}$ but not $t_{1}$ (\cref{rmk:two-vertices}). Find a path $ L_{1}'$ in $R^{o}$, the ridge of $F_{1}$ disjoint from $R$ and containing $t_{1}$, between $\pi_{R^{o}}^{F_{1}}(s_{1})$ and $ t_{1}$ and let $ L_{1}:=s_{1}\pi_{R^{o}}^{F_{1}}(s_{1}) L_{1}' t_{1}$.
 
{\bf The third step is to bring the $d-1$ terminal vertices $\hat x\in \Gamma_{12}$ into the facet $F_{12}$ so that they can be linked there, avoiding $s_{1}$.}  We consider two cases depending on the number of facets in $\St_{12}$.
  
\begin{figure} 
\includegraphics[scale=.9]{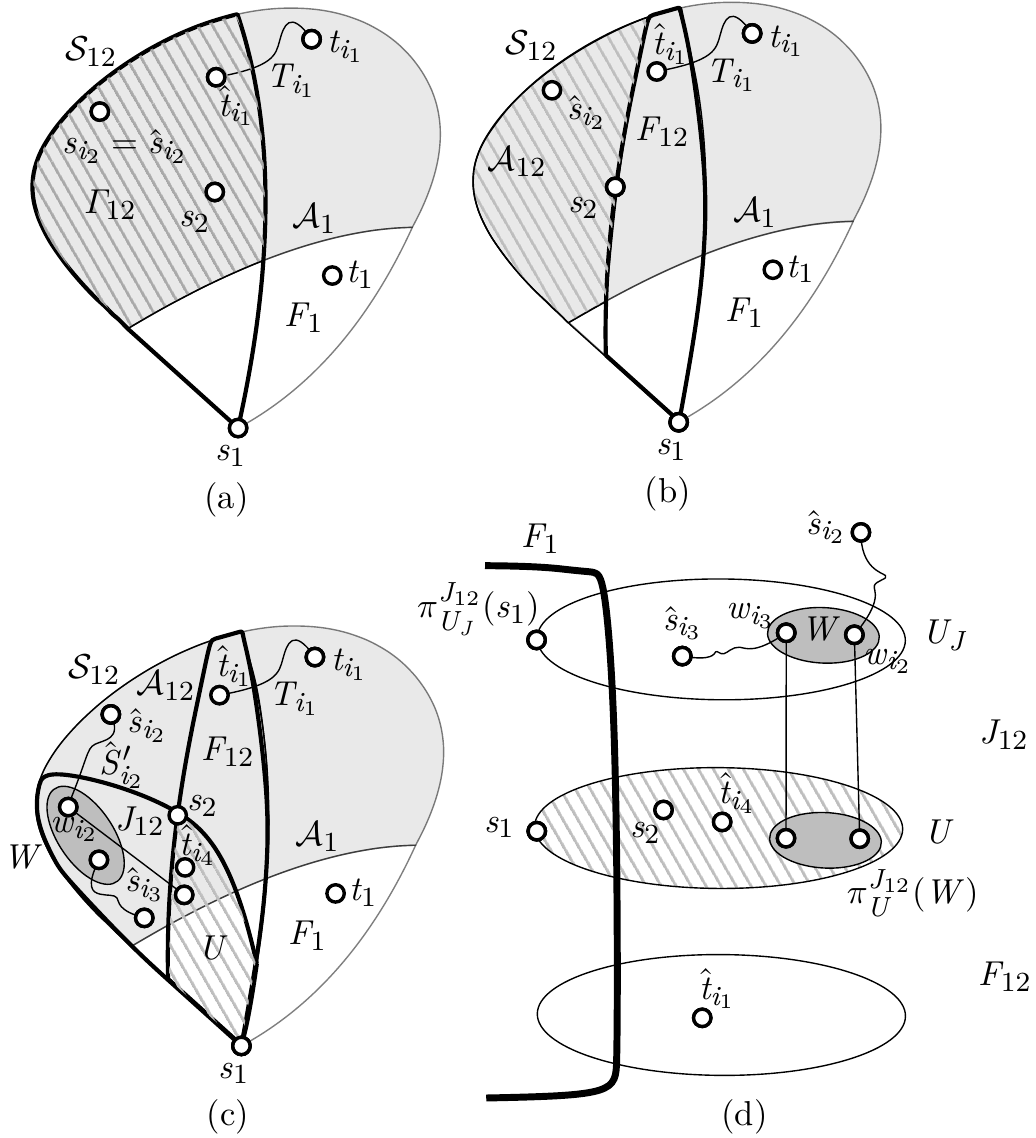}
\caption{Auxiliary figure for \cref{case:new-linkedness-thm-3} of \cref{lem:star-cubical}. A representation of $\St_{1}$. (a) A configuration where the subgraph $\Gamma_{12}$ is tiled in  falling pattern and the complex $\A_{1}$ is coloured in grey. (b) A depiction of $\St_{12}$ with more than one facet; the facet $F_{12}$ is highlighted in bold, the complex $\A_{1}$ is coloured in grey and the complex $\A_{12}$ is highlighted in falling pattern. (c)  A depiction of $\St_{12}$ with more than one facet; the facets  $F_{12}$ and $J_{12}$ are highlighted in bold and their intersection $U$ is highlighted in falling pattern; the set $W$ in $J_{12}$ is coloured in dark grey. (d) A depiction of a portion of $\St_{12}$, zooming in on the facets  $F_{12}$ and $J_{12}$; each facet is represented as the convex hull of two disjoint $(d-2)$-faces, and their intersection $U$  is highlighted in falling pattern. The sets $W$ and $\pi_{U}^{J_{12}}(W)$ in $J_{12}$ are coloured in dark grey.}\label{fig:Aux-Linked-Thm-Case4} 
\end{figure}
   
{\bf Suppose $\St_{12}$ only consists of $F_{12}$}. Then \[\hat X=\{\hat s_{2},\ldots,\hat s_{k},\hat t_{2},\ldots,\hat t_{k}\}\subset V(\Gamma_{12})\subset V(F_{12}).\] With the help of the strong $(k-1)$-linkedness of $F_{12}$ (\cref{thm:cube-strong-linkedness}), we can link the pairs $\{\hat s_{i},\hat t_{i}\}$ for $i\in [2,k]$ in $F_{12}$ through disjoint paths $\hat L_{i}$, all avoiding $s_{1}$.  The paths $\hat L_{i}$ concatenated with the paths $S_{i}$ and $ T_{i}$  for $i\in [2,k]$ give a $(Y\setminus \{s_{1}, t_{1}\})$-linkage $\{L_{2},\ldots,L_{k}\}$. Hence the desired $Y$-linkage is as follows.  
 
\[L_{i}:=\begin{cases}s_{1}\pi_{R^{o}}^{F_{1}}(s_{1}) L_{1}' t_{1},& \text{for $i=1$;}\\s_{i}S_{i}\hat s_{i}\hat L_{i}\hat t_{i} T_{i}t_{i},& \text{otherwise.}
 \end{cases}\]
  
{\bf Assume $\St_{12}$  has more than one facet}. We have that \[\hat X=\{\hat s_{2},\ldots,\hat s_{k},\hat t_{2},\ldots,\hat t_{k}\}\subset V(\Gamma_{12}).\] 
Define
\begin{itemize}
\item $\A_{12}$ as the complex of $\St_{12}$ induced by $V(\St_{12})\setminus (V(F_{1})\cup V(F_{12}))$. 
\end{itemize}
Then the graph $G(\A_{12})$ of $\A_{12}$  coincides with the subgraph of $\Gamma_{12}$ induced by $V(\Gamma_{12})\setminus V(F_{12})$. \Cref{fig:Aux-Linked-Thm-Case4}(b) depicts this configuration.

Our strategy is first to bring the $d-3$ terminal vertices $\hat x$ in $\Gamma_{12}$ other than $\hat s_{2}$ and $\hat t_{2}$ into $F_{12}\setminus F_{1}$ through disjoint paths $\hat S_{i}$ and $\hat T_{j}$, without touching $\hat s_{2}$ and $\hat t_{2}$. Second, denoting by $\tilde s_{i}$ and $\tilde t_{j}$ the intersection of $\hat S_{i}$ and $\hat T_{j}$ with $V(F_{12})\setminus V(F_{1})$, respectively, we link the pairs $\{\tilde s_{i},\tilde t_{i}\}$ for $i=[2,k]$ in $F_{12}$ through disjoint paths $\tilde L_{i}$, without touching $s_{1}$; here we resort to the strong $(k-1)$-linkedness of $F_{12}$. We develop these ideas below. 

From  \cref{lem:technical}(iii), it follows that $\A_{12}$ is nonempty and contains a spanning strongly connected $(d-3)$-subcomplex, thereby implying, by \cref{prop:connected-complex-connectivity}, that \[\text{$G(\A_{12})$ is $(d-3)$-connected.}\]
 Since $\St_{12}$ contains more than one facet, the following sets exist:
 \begin{itemize}
 \item $U$, a $(d-2)$-face  in $F_{12}$ that contains $s_{1}$ and $\hat s_{2}$ ($= s_{2}$)  (\cref{rmk:opposite-vertex-1});
 \item $J_{12}$, the other facet in $\St_{12}$ containing $U$;
 \item $U_{J}$, the $(d-2)$-face in $J_{12}$ disjoint from $U$, and as a consequence, disjoint from $F_{12}$; 
 \item $\C_{U}$,  the subcomplex of $\B(U)$ induced by $V(U)\setminus V(F_{1})$, namely the antistar of $U\cap F_{1}$ in $U$; and 
 
\item  $\C_{U_{J}}$, the subcomplex of $\B(U_{J})$ induced by $V(U_{J})\setminus  V(F_{1})$. 
 \end{itemize}
	
 The subcomplex $\C_{U}$ is nonempty, since $\hat s_{2}\in V(U)\setminus V(F_{1})$, and so, thanks to \cref{lem:cube-face-complex}, it is a strongly connected $(d-3)$-complex. Then, from $C_{U}$ containing a $(d-3)$-face it follows that
 \begin{equation}\label{eq:Cubical-Linkedness-Case4-U-Cardinality}
 |V(\C_{U})|=|V(U)\setminus V(F_{1}))|\ge 2^{d-3}\ge d-1\; \text{for $d\ge 5$}.
 \end{equation}
 
The subcomplex $\C_{U_{J}}$ is nonempty: if $U_{J}\cap F_{1}=\emptyset$ then $\C_{U_{J}}=\B(U_{J})$; otherwise $\C_{U_{J}}$ is the antistar of $U_{J}\cap F_{1}$ in $U_{J}$, and since $U\cap F_{1}\ne \emptyset$ ($s_{1}$ is in both), it follows that $U_{J}\not \subseteq F_{1}$. Put differently, the vertex in $J_{12}$ opposite to $s_{1}$ is not in $U$, since $s_{1}\in U$, nor is it in $F_{1}$, and so it must be in $\C_{U_{J}}$.  Therefore, according to \cref{lem:cube-face-complex}, $\C_{U_{J}}$ is a strongly connected $(d-3)$-complex.   Hence, in both instances,
 \begin{equation}\label{eq:Cubical-Linkedness-Case4-U_J-Cardinality}
  |V(\C_{U_{J}})|=|V(U_{J})\setminus V(F_{1}))|\ge 2^{d-3}\ge d-1 \; \text{for $d\ge 5$}.
 \end{equation}
 
Recall that we want to bring every vertex in the set $\hat X$, which is contained in $\Gamma_{12}$, into $F_{12}\setminus F_{1}$. We construct $|\hat X\cap V(\A_{12})|$ pairwise disjoint paths $\hat S_{i}$ and $\hat T_{j}$ from $\hat s_{i}\in \A_{12}$ and $\hat t_{j}\in \A_{12}$, respectively,  to $V(F_{12})\setminus V(F_{1})$ as follows.  Pick a set \[W\subset V(\C_{U_{J}})\setminus \pi_{U_{J}}^{J_{12}}\left((\hat X\cup\{s_{1}\})\cap U\right)\]
of $|\hat X\cap V(\A_{12})|$ vertices in $\C_{U_{J}}$. Then $\pi_{U}^{{J^{12}}}(W)$ is disjoint from $(\hat X\cup\{s_{1}\})\cap U$. In other words, the vertices in $W$ are in $\C_{U_{J}}$ and are not projections  of the vertices in $(\hat X\cup\{s_{1}\}) \cap U$ onto $U_{J}$.  We show that the set $W$ exists, which amounts to showing that $\C_{U_{J}}$ has enough vertices to accommodate $W$.

First note that 
\begin{equation}
\begin{gathered}\label{eq:Cardinality-hat-X-A}
|\hat X\cap V(\A_{12})|+|(\hat X\cup\{s_{1}\})\cap V(F_{12})|=|\hat X\cup \{s_{1}\}|=d,\\ 
(\hat X\cup\{s_{1}\})\cap V(U)\subseteq (\hat X\cup\{s_{1}\})\cap V(F_{12}).
\end{gathered}  
\end{equation} 

If $U_{J}\cap F_{1}=\emptyset$ then $\C_{U_{J}}=\B(U_{J})$. And \eqref{eq:Cardinality-hat-X-A} together with $|V(U_{J})|=2^{d-2}\ge d$ for $d\ge 5$ gives the following chain of inequalities
\begin{multline*}
\left|V(C_{U_{J}})\setminus \pi_{U_{J}}^{J_{12}}\left((\hat X\cup\{s_{1}\})\cap V(U)\right)\right|\ge d-\left|(\hat X\cup \{s_{1}\})\cap V(U)\right|\\
\ge\left|\hat X\cup\{s_{1}\}\right|-\left|(\hat X\cup \{s_{1}\})\cap V(F_{12})\right|
=\left|\hat X\cap V(\A_{12})\right|=\left|W\right|,
\end{multline*}
as desired. 

Suppose now $U_{J}\cap F_{1}\ne \emptyset$. Since $s_{1}\in U\cap F_{1}$ and $J_{12}=\conv \{U\cup U_{J}\}$, the cube $J_{12}\cap F_{1}$ has opposite facets $U_{J}\cap F_{1}$ and $U\cap F_{1}$. From $s_{1}\in U\cap F_{1}$ it follows that $\pi_{U_{J}}^{J_{12}}(s_{1})\in U_{J}\cap F_{1}$, and thus, that $\pi_{U_{J}}^{J_{12}}(s_{1})\not \in \C_{U_{J}}$; here we use the following remark.
\begin{remark}
Let $(K,K^{o})$ be opposite facets in a cube $Q$ and let $B$ be a proper face of $Q$ such that $B\cap K\ne \emptyset$ and $B\cap K^{o}\ne \emptyset$. Then $\pi^{Q}_{K^{o}}(B\cap K)=B\cap K^{o}$. 
\end{remark}
\noindent Since $\pi_{U_{J}}^{J_{12}}(s_{1})\not \in \C_{U_{J}}$, using \eqref{eq:Cubical-Linkedness-Case4-U_J-Cardinality} and \eqref{eq:Cardinality-hat-X-A} we get  	   
\begin{multline*} 
\left|V(C_{U_{J}})\setminus \pi_{U_{J}}^{J_{12}}\left((\hat X\cup\{s_{1}\})\cap V(U)\right)\right|=\left|V(C_{U_{J}})\setminus \pi_{U_{J}}^{J_{12}}\left(\hat X\cap V(U)\right)\right|\\ \ge d-1-\left|\hat X\cap V(U)\right|
\ge\left|\hat X\right|-\left|\hat X\cap V(F_{12})\right|
=\left|\hat X\cap V(\A_{12})\right|=\left|W\right|.
\end{multline*}
{\bf In this way, we have shown that $\C_{U_{J}}$ can accommodate the set $W$.} We now finalise teh case.
  
There are at most $d-3$ vertices $\hat x$ in $\hat X\cap V(\A_{12})$ because $\hat s_{2}$ and $\hat t_{2}$ are already in $V(F_{12})\setminus V(F_{1})$. Since $G(\A_{12})$ is $(d-3)$-connected, we can find $|W|=|\hat X\cap V(\A_{12})|$ pairwise disjoint paths $\hat S_{i}'$ and $\hat T_{j}'$ in $\A_{12}$ from the terminals $\hat s_{i}$ and $\hat t_{j}$ in $\hat X\cap V(\A_{12})$ to $W$.  The $\hat X$-valid path $\hat S_{i}$ from $\hat s_{i}\in \A_{12}$ to $V(F_{12})\setminus V(F_{1})$ then consists of the subpath $\hat S_{i}':=\hat s_{i}-w_{i}$ with $w_{i}\in W$  plus the edge $w_{i}\pi_{U}^{J_{12}}(w_{i})$; from the choice of $W$ it follows that $\pi_{U}^{J_{12}}(w_{i})\not \in \hat X\cup\{s_{1}\}$. The paths $\hat T_{j}'$ and $\hat T_{j}$ are defined analogously. \Cref{fig:Aux-Linked-Thm-Case4}(c)-(d) depicts this configuration. 

Denote by $\tilde s_{i}$ the intersection of $\hat S_{i}$ and $V(F_{12})\setminus V(F_{1})$; similarly, define  $\tilde t_{j}$. Every terminal vertex $\hat x$ already in $F_{12}$ is also denoted by $\tilde x$, and in this case we let $\hat S_{i}$ or $\hat T_{j}$ be the vertex $\tilde x$.  
 
Now  $F_{12}$ contains the pairs $\left\{\tilde s_{i},\tilde t_{i}\right\}$ for $i\in [2,k]$ and the terminal $s_{1}$, as desired. Link these pairs in $F_{12}$ through disjoint paths $\tilde L_{i}$, each avoiding $s_{1}$, with the use of the strong $(k-1)$-linkedness of $F_{12}$ (\cref{thm:cube-strong-linkedness}). The paths $\tilde L_{i}$ concatenated with the paths $S_{i}$, $\hat S_{i}$, $T_{i}$ and $\hat T_{i}$ for $i\in [2,k]$ give a $(Y\setminus \{s_{1}, t_{1}\})$-linkage $\{L_{2},\ldots,L_{k}\}$. Hence the desired $Y$-linkage is as follows.   
 
\[L_{i}:=\begin{cases}s_{1}\pi_{R^{o}}^{F_{1}}(s_{1}) L_{1}' t_{1},& \text{for $i=1$;}\\s_{i}S_{i}\hat s_{i}\hat S_{i}\tilde s_{i}\tilde L_{i}\tilde t_{i}\hat T_{i}\hat t_{i} T_{i} t_{i},& \text{otherwise.}
 \end{cases}\]

\begin{case}\label{case:new-linkedness-thm-4} $| X\cap V(F_{1})|= d+1$ and the vertex $s_{1}$ is not in Configuration $d$F. \end{case}

Here we have that $V(\A_{1})\cap  X=\emptyset$. This case is decomposed into three main subcases A, B and C, based on the nature of the vertex $s_{1}^{o}$ opposite to $s_{1}$ in $F_{1}$, which is the only vertex in $F_{1}$ that does not have an image under the injection from $F_{1}$ to $\A_{1}$ defined in \cref{lem:projections-star}. And each subcase is then analysed for $d\ge 7$ and $d=5$ separately.   The difficulty with $d=5$ stems from the $(d-2)$-faces   of the polytope not being 2-linked (\cref{cor:nonsimplicial-3polytope}).

\subsection *{\bf \uppercase{Subcase} A for $d\ge 7$. The vertex $s_{1}^{o}$ opposite  to $s_{1}$ in $F_{1}$ does not belong to $ X$.} Let $ X':= X\setminus \{ t_{1}\}$ and let $ Y':= Y\setminus \{\{s_{1}, t_{1}\}\}$. Since $| X'|=d$,  the strong $(k-1)$-linkedness of $F_{1}$ (\cref{thm:cube-strong-linkedness}) gives a $ Y'$-linkage $\{ L_{2},\ldots, L_{k}\}$ in the facet $F_{1}$ with each path $ L_{i}:= s_{i}- t_{i}$ ($i\in [2,k]$) avoiding $s_{1}$.  We find pairwise distinct  neighbours $s_{1}'$ and $ t_{1}'$ in $\A_1$ of $s_{1}$ and $ t_{1}$, respectively. If none of the paths $ L_{i}$ touches $ t_{1}$, we find a path $ L_{1}:=s_{1}- t_{1}$ in $\St_{1}$ that contains a subpath in $\A_1$  between  $s_{1}'$ and $ t_{1}'$ (here use the connectivity of $\A_{1}$, \cref{prop:star-minus-facet}), and we are home. Otherwise,  assume that the path $ L_{j}$ contains $ t_{1}$. With the help of \cref{lem:projections-star}, find pairwise distinct  neighbours $ s_{j}'$ and $ t_{j}'$ in $\A_1$ of $ s_{j}$ and $ t_{j}$, respectively, such that the vertices $s_{1}'$,  $ t_{1}'$, $ s_{j}'$ and $ t_{j}'$ are pairwise distinct. According to \cref{prop:star-minus-facet-linkedness}, the complex $\A_1$ is 2-linked for $d\ge 7$. Hence, we can find disjoint paths $ L_{1}':=s_{1}'- t_{1}'$ and $ L_{j}':= s_{j}'- t_{j}'$ in $\A_{1}$, respectively; these paths naturally give rise to paths $ L_{1}:=s_{1}s_{1}'L_{1}'t_{1}'t_{1}$ in $\St_1$  and $ L_{j}:= s_{j}s_{j}'L_{j}'t_{j}'t_{j}$ in $\St_1$. The paths $\left\{L_{1},\ldots,L_{k}\right\}$ give the desired $Y$-linkage.

\subsection *{\bf \uppercase{Subcase} B for $d\ge 7$. The vertex $s_{1}^{o}$ opposite  to $s_{1}$ in $F_{1}$ belongs to $ X$ but is different from $ t_{1}$, say $s_{1}^{o}= s_{2}$.}  First find a neighbour $s_{1}'$ of $s_{1}$ and a neighbour $ t_{1}'$ of $ t_{1}$ in $\A_1$. 
 There is a neighbour $ s_{2}^{F_{1}}$ of $ s_{2}$ in $F_{1}$ that is either $ t_{2}$ or a vertex not in $ X$: $\{s_{1}, s_{2}\}\cap N_{F_{1}}( s_{2})=\emptyset$ and $|N_{F_{1}}( s_{2})|=d-1$. The link $\Lk_1$ of $s_{1}$ in $F_{1}$ contains  all the vertices in $F_{1}$ except $s_{1}$ and $ s_{2}$. 

Suppose $ s_{2}^{F_{1}}= t_{2}$. Let $ L_{2}:= s_{2} t_{2}$, and using the $(k-1)$-linkedness of $\Lk_1$ (\cref{prop:link-cubical}), find  disjoint paths  $ t_{1}- t_{2}$ and $ L_{i}:= s_{i}- t_{i}$  for $i\in [3,k]$ in $\Lk_1$. Then define a path $ L_{1}:=s_{1}- t_{1}$ in $\St_{1}$ that contains a subpath in $\A_1$  between $s_{1}'$ and $ t_{1}'$; here we use the connectivity of $\A_{1}$ (\cref{prop:star-minus-facet}).  The paths $\left\{L_{1},\ldots,L_{k}\right\}$ give the desired $Y$-linkage. 

Assume $ s_{2}^{F_{1}}$ is not in $ X$. Observe that $| (X\setminus \{s_{1}, s_{2}\})\cup \{ s_{2}^{F_{1}}\}|=d$. Using the  $(k-1)$-linkedness of $\Lk_1$ for $d\ge 7$ (\cref{prop:link-cubical}), find in $\Lk_1$ disjoint paths  $ L_{2}':= s_{2}^{F_{1}}- t_{2}$ and $ L_{i}':= s_{i}- t_{i}$  for $i\in [3,k]$. Since $ t_{1}$ is also in $\Lk_{1}$ it may happen that it lies in one of the  paths $ L_{i}'$. If $ t_{1}$ does not belong to any of the paths $ L_{i}'$  for $i\in [2,k]$, then find a path $ L_{1}:=s_{1} s_{1}'L'_{1} t'_{1}t_{1}$ in $\St_{1}$ where $L_{1}'$ is a subpath in $\A_1$  between $s_{1}'$ and $ t_{1}'$, using the connectivity of $\A_{1}$ (\cref{prop:star-minus-facet}). In this scenario, let $L_{2}:= s_{2} s_{2}^{F_{1}} L_{2}' t_{2}$ and $L_{i}:=L_{i}'$ for $i\in [3,k]$; the desired $Y$-linkage is given by the paths $\left\{L_{1},\ldots,L_{k}\right\}$.

If $ t_{1}$ belongs to one of the paths $ L_{i}'$ with $i\in [2,k]$, say $ L_{j}'$, then consider in $\A_1$ a neighbour $ t_{j}'$ of $ t_{j}$ and, either a neighbour $ s_{j}'$ of $ s_{j}$ if $j\ne 2$ or a neighbour $ s_{2}'$ of $ s_{2}^{F_{1}}$. From \cref{lem:projections-star} it follows that the vertices $s_{1}'$, $ t_{1}'$,  $ s_{j}'$ and $ t_{j}'$ can be taken pairwise distinct. Since $\A_1$ is 2-linked for $d\ge 7$ (see~\cref{prop:star-minus-facet-linkedness}), find in $\A_1$ a path $ L_{1}'$   between $s_{1}'$ and $ t_{1}'$ and a path $ L_{j}''$ between $ s_{j}'$ and $ t_{j}'$. As a consequence, we obtain in $\St_{1}$ a path  $ L_{1}:=s_{1}s_{1}' L_{1}' t_{1}' t_{1}$ and, either a path $ L_{j}:= s_{j} s_{j}' L_{j}'' t_{j}' t_{j}$ if $j\ne 2$ or a  path $ L_{2}:= s_{2} s_{2}^{F_{1}} s_{2}' L_{2}'' t_{2}' t_{2}$. In addition, let $L_{i}:=L_{i}'$ for $i\in [3,k]$ and $i\ne j$. The paths $\left\{L_{1},\ldots,L_{k}\right\}$ give the desired $Y$-linkage.

\subsection *{\bf \uppercase{Subcases A and B} for $d=5$. The vertex $s_{1}^{o}$ opposite to $s_{1}$ in $F_{1}$ either does not belong to $ X$ or belongs to $ X$ but is different from $ t_{1}$.}

  Let $X:=\{s_{1},s_{2},s_{3}, t_{1},t_{2},t_{3}\}$ be any set of six vertices in the graph $G$ of a cubical $5$-polytope $P$. Also let $Y:=\{\{s_{1},t_{1}\},\{s_{2},t_{2}\},\{s_{3},t_{3}\}\}$. We aim to find a $Y$-linkage $\{L_{1},L_{2},L_{3}\}$ in $G$ where $L_{i}$ joins the pair $\{s_{i},t_{i}\}$ for $i=1,2,3$. 
    
In both subcases there is a 3-face $R$ of $F_{1}$ containing both $s_{1}$ and $ t_{1}$. Let $J_{1}$ be the other facet in $\St_{1}$ containing $R$. Denote by $R_{J}$ and $R_{F}$ the ridges in $J_{1}$ and $F_{1}$, respectively, that are disjoint from $R$. Then $s_{1}^{o}\in R_{F}$. We need the following claim.

\begin{claim}\label{cl:d=5} If a 3-cube contains three pairs of terminals, there must exist two pairs of terminals in the 3-cube, say $\{s_1,  t_{1}\}$ and $\{ s_{2}, t_{2}\}$,  that are not  arranged in the cyclic order $s_1, s_{2}, t_{1}, t_{2}$ in a 2-face of the cube.
\end{claim}
\begin{claimproof} If no terminal in the cube is in Configuration 3F, we are done. So  suppose that one is, say $s_{1}$, and that the sequence $s_1,x_{1},  t_{1},x_{2}$ of vertices of $X$ is present in cyclic order in a 2-face. Without loss of generality, assume that $ s_{2}\not\in\{x_{1},x_{2}\}$. Then $ s_{2}$ cannot be adjacent to both $ s_{1}$ and $ t_{1}$, since the bipartite graph $K_{2,3}$ is not a subgraph of $G(Q_{3})$ (\cref{rmk:cubical-common-neighbours}). Thus  the sequence $s_1, s_{2}, t_{1}, t_{2}$ cannot be in a 2-face in cyclic order. 
\end{claimproof}  
 
{\bf Suppose all the six terminals are in the 3-face $R$.} By virtue of \cref{cl:d=5},  we may assume that the pairs $\{s_{1},  t_{1}\}$ and $\{  s_{{2}}, t_{{2}}\}$ are not arranged in the cyclic order $s_1, s_{2}, t_{1}, t_{2}$ in a 2-face of $R$.  \cref{prop:3-polytopes}  ensures that the pairs $\{\pi_{R_{J}}^{J_{1}}( s_{1}),\pi_{R_{J}}^{J_{1}}( t_{1})\}$ and $\{\pi_{R_{J}}^{J_{1}}( s_{2}),\pi_{R_{J}}^{J_{1}}( t_{2})\}$ in $R_{J}$ can be linked in $R_{J}$ through disjoint paths $ L_{1}'$ and $ L_{2}'$, since the sequence $\pi_{R_{J}}^{J_{1}}( s_{1}),\pi_{R_{J}}^{J_{1}}( s_{2}),\pi_{R_{J}}^{J_{1}}(t_{1}),\pi_{R_{J}}^{J_{1}}( t_{2})$ cannot be in a 2-face of $R_{J}$ in cyclic order. Moreover, by the connectivity of $R_{F}$, there is a path $ L_{3}'$ in $R_{F}$ linking the  pair $\{\pi_{R_{F}}^{F_{1}}( s_{3}),\pi_{R_{F}}^{F_{1}}( t_{3})\}$. The linkage $\{ L_{1}',  L_{2}',  L_{3}'\}$ can naturally be extended to a $ Y$-linkage $\{ L_{1},  L_{2},  L_{3}\}$  as follows.

\[L_{i}:=\begin{cases}
	s_{i}\pi_{R_{J}}^{J_{1}}( s_{i})L_{i}'\pi_{R_{J}}^{J_{1}}(t_{i})t_{i}, &\text{for $i=1,2$};\\
	s_{3}\pi_{R_{F}}^{F_{1}}( s_{3})L_{3}'\pi_{R_{F}}^{F_{1}}(t_{3})t_{3}, &\text{otherwise}.
\end{cases}
\]

{\bf Suppose that $R$ contains a pair $\{ s_{i}, t_{i}\}$ for $i=2,3$, say $\{ s_{2},  t_{2}\}$}. There are at most five terminals in $R$, and consequently, applying  \cref{lem:short-distance} to the polytope $F_{1}$ and its facet $R$,  we obtain an $ X$-valid  path $ L_{1}:=s_{1}- t_{1}$  in $R$ or an $ X$-valid path  $ L_{2}:= s_{2}- t_{2}$  in $R$. For the sake of concreteness, say an $ X$-valid path $ L_{2}$ exists in $R$. From the connectivity of $R_{F}$ and $R_{J}$ follows the existence of a path $ L_{3}'$ in $R_{F}$ between $\pi_{R_{F}}^{F_{1}}( s_{3})$ and $\pi_{R_{F}}^{F_{1}}( t_{3})$, and of a path $ L_{1}'$ in $R_{J}$ between $\pi_{R_{J}}^{J_{1}}(s_{1})$ and $\pi_{R_{J}}^{J_{1}}( t_{1})$. The linkage $\{ L_{1}', L_{2}', L_{3}'\}$ can be extended to a linkage $\{s_{1}- t_{1}, s_{2}- t_{2}, s_{3}- t_{3}\}$ in $\St_{1}$.
	
{\bf Suppose that the ridge $R$ contains no other pair from $ Y$ and that the ridge $R_{F}$ contains a  pair $( s_{i}, t_{i})$ ($i=2,3$)}.  Without loss of generality, assume $ s_{2}$ and $ t_{2}$ are in $R_{F}$. 

First suppose that $s_{3}\in R$, which implies that $ t_{3}\in R_{F}$. Further suppose that there is a path $T_{3}$ of length at most two from $t_{3}$ to $R$ that is disjoint from $ X\setminus\{ s_{3}, t_{3}\}$. Let $ \{ t_{3}'\}:=V(T_{3})\cap V(R)$. Use the 2-linkedness of $J_{1}$ (\cref{prop:4polytopes}) to find disjoint paths $ L_{1}:=s_{1}- t_{1}$ and $ L_{3}':= s_{3}- t_{3}'$ in $J_{1}$. Let $ L_{3}:= s_{3} L_{3}' t_{3}'T_{3} t_{3}$. Use the 3-connectivity of $R_{F}$ to find an $ X$-valid path $ L_{2}:= s_{2}- t_{2}$ in $R_{F}$ that is disjoint from $V(T_{3})$; note that $|V(T_{3})\cap V(R_{F})|\le 2$. The paths $\{ L_{1}, L_{2}, L_{3}\}$ give the desired  $Y$-linkage. Now suppose there is no such path $T_{3}$ from $ t_{3}$ to $R$. Then, the projection $\pi_{R}^{F_{1}}( t_{3})$ is in  $\{s_{1}, t_{1}\}$, say $\pi_{R}^{F_{1}}( t_{3})=t_{1}$; the projection $\pi_{R_{F}}^{F_{1}}(s_{1})$ is a neighbour of $ t_{3}$ in $R_{F}$; and both $ s_{2}$ and $ t_{2}$ are neighbours of $ t_{3}$ in $R_{F}$. This configuration implies that $s_{1}$ and $ t_{1}$ are adjacent in $R$. Let $ L_{1}:=s_{1} t_{1}$. Find a path $ L_{2}:= s_{2}- t_{2}$ in $R_{F}$ that is disjoint from $t_{3}$, using the 3-connectivity of $R_{F}$.   
Then find a neighbour $ s_{3}'$ in $\A_{1}$ of $ s_{3}$ and a neighbour $ t_{3}'$ in $\A_{1}$ of $t_{3}$; note that, since $\dist_{F_{1}}(s_{1}, t_{3})\le 2$, we have that $ t_{3}\ne s_{1}^{o}$. Find a path $ L_{3}$ in $\St_{1}$ between $ s_{3}$ and $ t_{3}$  that contains a subpath $L_{3}'$ in $\A_{1}$ between $ s_{3}'$ and $ t_{3}'$; here use the connectivity of $\A_{1}$ (\cref{prop:star-minus-facet}): $L_{3}:=s_{3}s_{3}'L_{3}'t_{3}'t_{3}$. The linkage $\{ L_{1}, L_{2}, L_{3}\}$ is the desired  $Y$-linkage.

Assume that $ s_{3}\in R_{F}$; by symmetry we can further assume that $ t_{3}\in R_{F}$. The connectivity of $R$ ensures the existence of a path $ L_{1}:= s_{1}- t_{1}$ therein.  In the case of $s_{1}^{o}\in  X$, without loss of generality, assume $s_{1}^{o}= s_{2}$. The 3-connectivity of $R_{F}$ ensures the existence of an $ X$-valid path $L_{2}:= s_{2}- t_{2}$ therein. Use \cref{lem:projections-star} to find pairwise distinct neighbours $ s_{3}'$ of $ s_{3}$ and $ t_{3}'$ of $ t_{3}$ in $\A_{1}$; these exist since $ s_{3}\ne s_{1}^{o}$ and $ t_{3}\ne s_{1}^{o}$. Using the connectivity of $\A_{1}$ (\cref{prop:star-minus-facet}), find a path $ L_{3}:= s_{3}- t_{3}$ in $\St_{1}$ that contains a subpath $ s_{3}'- t_{3}'$ in $\A_{1}$. The linkage $\{ L_{1}, L_{2}, L_{3}\}$ is the desired  $Y$-linkage.

{\bf Assume neither $R$ nor $R_{F}$ contains a pair $\{ s_{i}, t_{i}\}$ ($i=2,3$)}. Without loss of generality, assume  that $ s_{2}, s_{3}\in R$,  that $ t_{2}, t_{3}\in R_{F}$ and that $ t_{2}\ne s_{1}^{o}$.  

First suppose  that there exists a path $S_{3}$ in $F_{1}$ from $ s_{3}$ to $R_{F}$ that is  of length at most two and is disjoint from $ X\setminus \{ s_{3}, t_{3}\}$. Let $\{\hat s_{3}\}:=V(S_{3})\cap V(R_{F})$.   Find pairwise distinct neighbours $ s_{2}'$ and $ t_{2}'$ of $ s_{2}$ and $ t_{2}$, respectively, in $\A_{1}$. And find a path $ L_{2}:= s_{2}- t_{2}$ in $\St_{1}$ that contains a subpath $ s_{2}'- t_{2}'$ in $\A_{1}$ (using the connectivity of $\A_{1}$). 
Using the 3-connectivity of $R_{F}$ link the pair $\{\hat s_{3}, t_{3}\}$ in $R_{F}$ through a path $ L_{3}'$ that is disjoint from $t_{2}$. Let $L_{3}:=s_{3}S_{3}\hat s_{3}L_{3}'t_{3}$. Since \cref{cor:separator-independent}  ensures that any separator of size three in a 3-cube must be independent, we can find a path $ L_{1}:=s_{1}- t_{1}$ in $R$ that is disjoint from $s_{2}$ and $V(S_{3})\cap V(R)$; the set $V(S_{3})\cap V(R)$ has either cardinality one or contains an edge. The paths $\{ L_{1}, L_{2}, L_{3}\}$ form the desired  $Y$-linkage.

Assume that there is no such path $S_{3}$. In this case, the neighbours of $ s_{3}$ in $F_{1}$ are $s_{1}, t_{1}, s_{2}$ from $R$ and $ t_{2}$ from $R_{F}$. Use \cref{lem:projections-star} to find a neighbour  $ s_{3}'$ of $ s_{3}$ in $\A_{1}$. Again use \cref{lem:projections-star}  either to find  a neighbour $ t_{3}'$ of $ t_{3}$ if $ t_{3}\ne s_{1}^{o}$ or to find a neighbour $ t_{3}'$ of a neighbour $u$ of $ t_{3}$ in $R_{F}$ (with $u\ne t_{2}$) if $  t_{3}=s_{1}^{o}$.  Let $T_{3}$ be the path of length at most two from $ t_{3}$ to $\A_{1}$ defined as $T_{3}= t_{3} t_{3}'$ if $ t_{3}\ne s_{1}^{o}$ and $T_{3}= t_{3}u t_{3}'$ if $ t_{3}= s_{1}^{o}$. Find  a path $ L_{3}$ in $\St_{1}$ between $ s_{3}$ and $ t_{3}$  that contains a subpath in $\A_{1}$ between $ s_{3}'$ and $ t_{3}'$; here use the connectivity of $\A_{1}$ (\cref{prop:star-minus-facet}). We next find a path $S_{2}$ in $F_{1}$ from $ s_{2}$ to $R_{F}$ that is  of length at most two and is disjoint from $V(T_{3})\cup \{s_{1}, t_{1}, s_{3}\}$. There are exactly four disjoint such  $ s_{2}-R_{F}$ paths of length at most two, one through each of the neighbours of $ s_{2}$ in $F_{1}$. One such path is $ s_{2} s_{3} t_{2}$. Among the remaining three $ s_{2}-R_{F}$ paths, since none of them contains $s_{1}$ or $ t_{1}$ and since $|V(T_{3})\cap V(R_{F})|\le 2$, we find the path $S_{2}$.  Let $\hat s_{2}:=V(S_{2})\cap V(R_{F})$. Find a path $ L_{2}':=\hat s_{2}- t_{2}$ in $R_{F}$ that is disjoint from $V(T_{3})$, using the 3-connectivity of $R_{F}$. Let $ L_{2}:= s_{2}S_{2}\hat s_{2} L_{2}' t_{2}$. Since the vertices in $(V(S_{2})\cap V(R))\cup\{ s_{3}\}$ cannot separate $s_{1}$ from $ t_{1}$ in $R$ (\cref{cor:separator-independent}), find a path $ L_{1}:=s_{1}- t_{1}$ in $R$ disjoint from $V(S_{2})\cap V(R)\cup\{ s_{3}\}$; the set $V(S_{2})$ has cardinality one or contains one edge. The paths $\{ L_{1}, L_{2}, L_{3}\}$ form the desired  $Y$-linkage.

\subsection *{\bf \uppercase{Subcase} C  for $d\ge 7$. The vertex opposite to $s_{1}$ in $F_{1}$ coincides with $ t_{1}$. And the vertex $s_{1}$ is not in Configuration $d$F.}  Then $ t_{1}$ has no neighbour in $\A_1$. In fact, $F_{1}$ is the only facet in $\St_{1}$ containing $ t_{1}$. 

Because the vertex $s_{1}$ is not in Configuration $d$F, $ t_{1}$ has a neighbour $ t_{1}^{F_{1}}$ in $F_{1}$ that is not in $ X$. Here we  reason as in the scenario in which $ s_{2}=s_{1}^{o}$ and $ s_{2}$ has a neighbour not in $ X$. 

First, using the $(k-1)$-linkedness of $\Lk_1$ (\cref{prop:link-cubical})  find  disjoint paths $ L_{i}:= s_{i}- t_{i}$  in $\Lk_1$ for $i\in [2,k]$. It may happen that $t_{1}^{F_{1}}$ is in one of the paths $L_{i}$ for $i\in [2,k]$. Second, consider neighbours  $s_{1}'$ and $ t_{1}'$ in $\A_1$ of $s_{1}$ and $ t_{1}^{F_{1}}$, respectively. 

 If $ t_{1}^{F_{1}}$ doesn't belong to any path $ L_{i}$, then  find a path $ L_{1}:=s_{1}- t_{1}$ that contains the edge $ t_{1} t_{1}^{F_{1}}$ and a subpath $L_{1}'$ in $\A_1$  between $s_{1}'$ and $ t_{1}'$; that is,  $L_{1}=s_{1}s_{1}'L_{1}'t_{1}'t_{1}^{F_{1}}t_{1}$. The desired $Y$-linkage is given by  $\{ L_{1}, \ldots, L_{k}\}$.

 If $ t_{1}^{F_{1}}$ belongs to one of the paths $ L_{i}$ with $i\in [2,k]$, say $ L_{j}$, then disregard this path $L_{j}$ and consider in $\A_1$ a neighbour $ s_{j}'$ of $ s_{j}$ and a neighbour $ t_{j}'$ of $ t_{j}$. From \cref{lem:projections-star}, it follows that the vertices $s_{1}'$, $ t_{1}'$, $ s_{j}'$ and $ t_{j}'$ can be taken pairwise distinct. Using the 2-linkedness of $\A_{1}$ for $d\ge 7$, find a path $L_{1}'$ in $\A_{1}$ between $s_{1}'$ and $ t_{1}'$ and a path $ L_{j}'$ in $\A_1$ between $ s_{j}'$ and $ t_{j}'$. Let $L_{1}:=s_{1}s_{1}'L_{1}'t_{1}'t_{1}^{F_{1}}t_{1}$ and let $L_{j}:=s_{j}s_{j}'L_{j}'t_{j}'t_{j}$ be the new $s_{j}-t_{j}$ path.  The paths  $\{ L_{1}, \ldots, L_{k}\}$ form the desired $Y$-linkage.

\subsection *{\bf \uppercase{Subcase} C for $d=5$. The vertex opposite to $s_{1}$ in $F_{1}$ coincides with $ t_{1}$. And the vertex $s_{1}$ is not in Configuration $d$F.}  
 
Hence we may suppose  that $ t_{1}$ has a neighbour $ t_{1}'$ not in $ X$. We reason as in Subcases A and B for $d=5$. We give the details for the sake of completeness.  

Let $R$ denote the $3$-face in $F_{1}$ containing both $s_{1}$ and $ t_{1}'$; $\dist_{R}(s_{1}, t_{1}')=3$. Let $R_{F}$ be the $3$-face of $F_{1}$ disjoint from $R$.  Let $J_{1}$ be the other facet in $\St_{1}$ containing $R$ and let $R_{J}$ be the $3$-face of $J_{1}$ disjoint from $R$.
  
{\bf Suppose $R$ contains a pair $\{ s_{i}, t_{i}\}$ ($i=2,3$), say $( s_{2}, t_{2})$.} There are at most five terminals in $R$. Since the smallest face in $R$ containing $s_{1}$ and $ t_{1}'$ is 3-dimensional,  the sequence  $\pi_{R_{J}}^{J_{1}}(s_{1}), \pi_{R_{J}}^{J_{1}}( s_{2}), \pi_{R_{J}}^{J_{1}}( t_{1}'),\pi_{R_{J}}^{J_{1}}( t_{2})$ cannot appear in a 2-face of $R_{J}$ in cyclic order. As a consequence, the pairs  $\{\pi_{R_{J}}^{J_{1}}(s_{1}), \pi_{R_{J}}^{J_{1}}( t_{1}')\}$ and $\{\pi_{R_{J}}^{J_{1}}( s_{2}), \pi_{R_{J}}^{J_{1}}( t_{2})\}$ can be linked in $R_{J}$ through disjoint paths $ L_{1}'$ and $ L_{2}'$, thanks to \cref{prop:3-polytopes}. Let $ L_{1}:=s_{1}\pi_{R_{J}}^{J_{1}}(s_{1}) L_{1}'\pi_{R_{J}}^{J_{1}}( t_{1}') t_{1}' t_{1}$ and $ L_{2}:= s_{2}\pi_{R_{J}}^{J_{1}}(s_{2}) L_{2}'\pi_{R_{J}}^{J_{1}}( t_{2}) t_{2}$. From the 3-connectivity of $R_{F}$ follows the existence of a path $ L_{3}'$ in $R_{F}$ between $\pi_{R_{F}}^{F_{1}}( s_{3})$ and $\pi_{R_{F}}^{F_{1}}( t_{3})$ that avoids $ t_{1}$. Let $ L_{3}:= s_{3}\pi_{R_{F}}^{F_{1}}( s_{3}) L_{3}'\pi_{R_{F}}^{F_{1}}( t_{3}) t_{3}$. The paths $\{ L_{1}, L_{2}, L_{3}\}$ form the desired   $ Y$-linkage.

{\bf Suppose that the ridge $R$ contains no pair  $\{ s_{i}, t_{i}\}$ ($i=2,3$) and that the ridge $R_{F}$ contains a  pair $\{ s_{i}, t_{i}\}$ ($i=2,3$), say $\{ s_{2}, t_{2}\}$}. Then, there are at most five terminals in $R_{F}$. If there are at most four terminals in $R_{F}$, the 3-connectivity of $R_{F}$ ensures the existence of an $ X$-valid path $ L_{2}:= s_{2}- t_{2}$ in $R_{F}$; if there are exactly five terminals in $R_{F}$, applying \cref{lem:short-distance} to the polytope $F_{1}$ and its facet $R_{F}$ gives either an $ X$-valid path $ L_{2}:= s_{2}- t_{2}$ or an $ X$-valid path  $ L_{3}:= s_{3}- t_{3}$ in $R_{F}$. As a result, regardless of the number of terminals in $R_{F}$, we can assume there is an $ X$-valid path $ L_{2}:= s_{2}-t_{2}$ in $R_{F}$. Find pairwise distinct neighbours $ s_{3}'$ and $ t_{3}'$  in $\A_{1}$ of $ s_{3}$ and $ t_{3}$, respectively, and a path $ L_{3}$ in $\St_{1}$ between $ s_{3}$ and $ t_{3}$  that contains a subpath in $\A_{1}$ between $ s_{3}'$ and $ t_{3}'$; here use the connectivity of $\A_{1}$ (\cref{prop:star-minus-facet}). In addition, let $ L_{1}'$ be a path in $R$ between  $s_{1}$ and $ t_{1}'$;  here use the 3-connectivity of $R$ to avoid any terminal in $R$. Let $L_{1}:=s_{1}L_{1}'t_{1}'t_{1}$.  The  $Y$-linkage is given by the paths $\{ L_{1}, L_{2}, L_{3}\}$.

{\bf Assume neither $R$ nor $R_{F_{1}}$ contains a pair $\{ s_{i}, t_{i}\}$ ($i=2,3$).}  Without loss of generality,  we can assume $ s_{2}, s_{3}\in R$ and $ t_{2}, t_{3}\in R_{F}$. 

For {\it some} $i=2,3$, there exists a path $S_{i}$ in $F_{1}$ from $ s_{i}$ to $R_{F}$  that is of length at most two and is disjoint from $ t_{1}'$ and $ X\setminus \{ s_{i}, t_{i}\}$. Suppose there is no such path $S_{3}=s_{3}-R_{F}$. Then the neighbours of $ s_{3}$ in $F_{1}$ would be $s_{1}, t_{1}', s_{2}$ from $R$ and $ t_{2}$ from $R_{F}$. But, since there are exactly four $ s_{2}-R_{F}$ paths of length at most two in $F_{1}$ and since the vertex $ s_{2}$ could not be adjacent to $\{s_{1}, t_{1}'\}$, the existence of such a  path  $S_{2}= s_{2}-R_{F}$ would be guaranteed.  Hence assume the existence of such a path $S_{3}= s_{3}-R_{F}$. Let $\{\hat s_{3}\}:=V(S_{3})\cap V(R_{F})$. Find an $ X$-valid path $ L_{3}':=\hat s_{3}- t_{3}$ in $R_{F}$ using its 3-connectivity. Let $ L_{3}:= s_{3}S_{3}\hat s_{3} L_{3}' t_{3}$.  Then find neighbours $ s_{2}'$ and $ t_{2}'$ of $ s_{2}$ and $ t_{2}$, respectively, in $\A_{1}$, and a path $ L_{2}:= s_{2}- t_{2}$ in $\St_{1}$ that contains a subpath $ s_{2}'- t_{2}'$ in $\A_{1}$ (using the connectivity of $\A_{1}$). Since \cref{cor:separator-independent}  ensures that any separator of size three in a 3-cube must be independent, we can find an $ L_{1}':=s_{1}- t_{1}'$ in $R$ that is disjoint from $ s_{2}$ and $V(S_{3})\cap V(R)$; the set $V(S_{3})\cap V(R)$ has either cardinality one or contains an edge. Let  $ L_{1}:=s_{1} L_{1}' t_{1}' t_{1}$. The paths $\{ L_{1}, L_{2}, L_{3}\}$ form the desired   $ Y$-linkage.

And finally, the proof  of \cref{lem:star-cubical} is complete.
\end{proof}
	
\section{Strong linkedness of cubical polytopes}
The property of strong linkedness, see \cref{thm:4polytopes-strong-linkedness,thm:cube-strong-linkedness}, also holds for cubical polytopes.

\begin{theorem}[Strong linkedness of cubical polytopes]\label{thm:cubical-strong-linkedness} For every $d\ne 3$, a cubical $d$-polytope is strongly $\floor{(d+1)/2}$-linked.\end{theorem} 
 
\begin{proof} Let $P$ be a cubical $d$-polytope. For odd $d$ \cref{thm:cubical-strong-linkedness,thm:cubical} are equivalent. So assume $d=2k$. Let $X$ be a set of $d+1$ vertices in $P$. Arbitrarily pair $2k$ vertices in $X$ to obtain $Y:=\{\{s_{1},t_{1}\},\ldots,\{s_{k},t_{k}\}\}$. Let $x$ be the vertex of $X$ not paired in $Y$. We find a $Y$-linkage $\{L_{1},\ldots, L_{k}\}$ where each path $L_{i}$ joins the pair $\{s_{i},t_{i}\}$ and avoids the vertex $x$.  

Using the $d$-connectivity of $G(P)$ and Menger's theorem, bring the $d=2k$ terminals in $X\setminus \{x\}$ to the link of $x$ in the boundary complex of $P$ through $2k$ disjoint paths $L_{s_{i}}$ and $L_{t_{i}}$ for $i\in [1,k]$. Let $s_{i}':=V(L_{s_{i}})\cap \lk(x)$  and $t_{i}':=V(L_{t_{i}})\cap \lk(x)$ for $i\in [1,k]$. Thanks to \cref{prop:link-polytope}, the link of $x$ is combinatorially equivalent to a cubical $(d-1)$-polytope, which is $d/2$-linked by \cref{thm:cubical}. Using the $d/2$-linkedness of $\lk(x)$, find disjoint paths $L_{i}':=s_{i}'-t_{i}'$ in $\lk(x)$. Observe that all these $k$ paths $\{L_{1}',\ldots,L_{k}'\}$ avoid $x$. Extend each path $L_{i}'$ with $L_{s_{i}}$ and $L_{t_{i}}$ to form a path $L_{i}:=s_{i}-t_{i}$ for $i\in [1,k]$. The paths $\{L_{1},\ldots,L_{k}\}$ forms the desired $Y$-linkage.
\end{proof}	
\section{Acknowledgments} The authors would like to thank the anonymous referees for their detailed comments and suggestions, including a new proof for \cref{lem:facets-association}. The presentation of the paper has greatly benefited from the referees' input.

%\section{Concluding remarks}

%---REFERENCES---

\providecommand{\bysame}{\leavevmode\hbox to3em{\hrulefill}\thinspace}
\providecommand{\MR}{\relax\ifhmode\unskip\space\fi MR }
\providecommand{\MRhref}[2]{%
  \href{http://www.ams.org/mathscinet-getitem?mr=#1}{#2}
}
\providecommand{\href}[2]{#2}

\end{document}